\documentclass[twoside,leqno]{article}
\usepackage[letterpaper]{geometry}
\usepackage{tikz,tikz-cd}
\usetikzlibrary{arrows.meta}
\usetikzlibrary{decorations.markings,calc}
\usepackage{bm}

\usepackage{siamproceedings}

\usepackage[T1]{fontenc}
\usepackage{amsfonts,amssymb}
\usepackage{graphicx}
\usepackage{epstopdf}
\usepackage{enumitem}
\usepackage{algorithmic}
\ifpdf
  \DeclareGraphicsExtensions{.eps,.pdf,.png,.jpg}
\else
  \DeclareGraphicsExtensions{.eps}
\fi

\newcommand{\longdashrightarrow}[1][2.5pt]{%
  \settowidth{\@tempdima}{$\longrightarrow$}\longrightarrow
  \makebox[-\@tempdima]{\hskip-0.5ex\color{white}\rule[0.5ex]{#1}{1pt}}
    \phantom{\longrightarrow}
  \makebox[-\@tempdima]{\hskip-2.8ex\color{white}\rule[0.5ex]{#1}{1pt}}
  \phantom{\longrightarrow}
}

\newsiamremark{remark}{Remark}
\newsiamremark{example}{Example}
\newsiamremark{hypothesis}{Hypothesis}
\crefname{hypothesis}{Hypothesis}{Hypotheses}
\newsiamthm{claim}{Claim}
\newsiamthm{conjecture}{Conjecture}

\def\A{{\mathcal{A}}}
\def\I{{\mathcal{I}}}
\def\O{{\mathcal{O}}}
\def\L{{\mathcal{L}}}
\def\a{{\mathbf{a}}}
\def\x{{\mathbf{x}}}

\def\R{{\mathbb{R}}}
\def\C{{\mathbb{C}}}

\def\P{{\mathbb{P}}}
\def\Id{{\rm Id}}

\def\tla{{\tilde \lambda}}
\def\p{{\mathbf{p}}}
\def\q{{\mathbf{q}}}
\def\y{{\mathbf{y}}}

\def\M{{\mathcal{M}}}
\def\bM{\overline{\M}}
\def\T{{\mathcal{T}}}

\def\WW{{\mathcal{W}}}
\def\MS{{\C^4}}
\def\Int{{\rm Int}}
\def\Fl{{\rm Fl}}

\def\tF{{\tilde F}}

\def\OS{{A}}

\def\atom{c}
\def\Ccal{{\mathcal{C}}}
\def\bl{{\bm{\ell}}}
\def\mini{\fontsize{7pt}{8pt}\selectfont}

\def\tree{{\rm tree}}
\usepackage{amsopn}

\DeclareMathOperator{\Tr}{Tr}
\DeclareMathOperator{\Gr}{Gr}
\DeclareMathOperator{\SO}{SO}
\DeclareMathOperator{\SU}{SU}
\DeclareMathOperator{\SL}{SL}
\DeclareMathOperator{\PGL}{PGL}
\DeclareMathOperator{\Conf}{Conf}
\DeclareMathOperator{\Res}{Res}
\DeclareMathOperator{\pt}{pt}
\DeclareMathOperator{\Ass}{Ass}
\DeclareMathOperator{\Crit}{Crit}
\DeclareMathOperator{\reg}{reg}
\DeclareMathOperator{\Trop}{Trop}

\DeclareMathOperator{\Vol}{Vol}
\DeclareMathOperator{\conv}{conv}
\DeclareMathOperator{\rank}{rank}

\newcommand\ip[1]{{\langle #1 \rangle}}

\newcommand\gdRip[1]{\langle #1 \rangle^{\nabla}}

\begin{document}

\title{\Large The combinatorial geometry of particle physics}
    \author{Thomas Lam \thanks{University of Michigan, Ann Arbor, USA (\email{tfylam@umich.edu}).}
 }

\date{}
\maketitle







\begin{abstract} 
Recent breakthroughs in the study of scattering amplitudes have uncovered profound and unexpected connections with combinatorial geometry. These connections range from classical structures -- such as polytopes, matroids, and Grassmannians -- to more modern developments including positroid varieties and the amplituhedron. Together they point toward the unifying framework of positive geometry, in which geometric domains canonically determine analytic functions governing scattering processes. This survey traces the emergence of positive geometry from the physics of amplitudes, building towards recent progress on amplitudes for matroids.
\end{abstract}

\section{Introduction.}\label{sec:intro}
Quantum field theory (QFT) is the mathematical formalism of particle physics, used to model the results of particle accelerator experiments, such as those at the Large Hadron Collider.  Despite its central role, many aspects of QFT are still mysterious, drawing sustained interest from mathematicians.  Over the past two decades, the study of scattering amplitudes, functions that encode probabilities of outcomes in elementary particle scattering, has undergone a remarkable transformation.  Calculations once dominated by elaborate Feynman diagram expansions have given way to new approaches that expose hidden simplicity and symmetry. These insights have brought amplitudes into direct dialogue with combinatorial geometry, revealing deep connections to polytopes, Grassmannians, and matroids, and beyond.

To set the stage, we start in \cref{sec:amplitudes,sec:gauge} with a quick overview of the physics.  In \cref{sec:amplitudes}, we give an informal introduction to the perturbative scattering amplitudes via Feynman diagrams.   In \cref{sec:gauge}, we recall how planarity appears in gauge theory amplitudes, and discuss the landmark Parke-Taylor formula for gluon scattering.  

Combinatorial geometry arises naturally: imposing the condition that particles are massless leads to the intersection theory for lines in projective space (\cref{sec:Schubert}).  The classical Schubert problem involving four given lines in three-space emerges from the Feynman box diagram.  More involved Feynman diagrams leads to the notion of on-shell diagram (\cref{sec:positroid}).  These diagrams directly give rise to the positroid stratification of the Grassmannian \cite{KLS,PosTP}.  By the seminal works \cite{BCFW,Grassbook}, tree-level scattering amplitudes for (super) Yang-Mills (SYM) theory can be completely described in terms of the combinatorics and geometry of positroid varieties.

In \cref{sec:ampli}, we introduce the amplituhedron \cite{AT}, a geometric incarnation of the SYM scattering amplitude as a ``polytope’’ in the Grassmannian.  In \cref{sec:positive}, we introduce the notion of positive geometries \cite{ABL}.  This mathematical construct canonically produces functions (such as the amplitude) from geometries (such as the amplituhedron).

In \cref{sec:polytope}, we turn to the most familiar example of positive geometries: polytopes.  The notion of canonical form for positive geometries suggests the definition of the dual mixed volume \cite{GLX}; the planar $\phi^3$-amplitude arises as the dual mixed volume of an associahedron.  While the planar $\phi^3$-amplitude is a toy model in quantum field theory, it is a ``low-energy’’ limit of the open string theory amplitude.  The latter function is intricately tied to the combinatorial geometry of the moduli spaces $\M_{g,n}$ of genus $g$ curves with $n$ marked points.  In \cref{sec:moduli}, we discuss aspects of the positive geometry of the moduli space $\M_{0,n}$, a concise excerpt of our lecture notes \cite{LamModuli}.

In \cref{sec:hyper,sec:matroid,sec:tropical}, we move from $\M_{0,n}$ to matroids.  We explain how the scattering equations formalism \cite{CHYarbitrary} for amplitudes can be generalized to the setting of hyperplane arrangement complements \cite{LamMat}.  Well studied notions from hyperplane arrangements such as the Orlik-Solomon algebra, Aomoto cohomology, and the Schechtman-Varchenko bilinear form make an appearance.  Our joint work with Eur \cite{EL} shows that topes of oriented matroids can be viewed as positive geometries, which leads to a notion of amplitudes for matroids.  We explain the basic analogy between matroids and QFT discovered in \cite{LamMat}.  In \cref{sec:tropical}, we give a formula for matroid amplitudes using the geometry of the positive Bergman fan, ending our story with tropical geometry.

For other surveys of this rapidly developing field, see \cite{LamPosGeom,LamModuli,FeSa,RaStTe}.

\section{Scattering amplitudes.}\label{sec:amplitudes}
Scattering amplitudes are functions of the momenta of the particles involved.  We write
$\A_n(\p_i)= \A_n(\p_1,\ldots,\p_n)$ for a scattering amplitude of $n$-particles with momentum vectors $\p_1,\ldots,\p_n \in \R^D$, where $D$ is the dimension of space-time.  For the physics of the real world, we would take $D = 4$, so that $\p_i$ are four-dimensional momentum vectors.  Hidden in this notation is the ``type'' (electron, photon, and so on) of the $n$ particles being scattered.

The probability of witnessing a particular scattering event is proportional to the integral
$$
\int |\A_n(\p_i)|^2 d\mu 
$$
for an appropriate measure $d\mu$.  From a theoretical perspective, the amplitude $\A_n(\p_i) $ can be expressed as a \emph{path integral} in quantum field theory.  In practice, one typically considers a perturbative expansion
\begin{equation}\label{eq:pert}
\A_n(\p_i) = \sum_{L=0}^\infty g^L \A^{(L)}_n(\p_i),
\end{equation}
where the sum is over the \emph{loop order} $L$, and $g$ is the coupling constant, measuring the strength of the interaction of particles.  When $g$ is small, one hopes that the perturbative expansion \eqref{eq:pert} holds.

The term $\A^{(L)}_n(\p_i)$ is the \emph{$L$-loop amplitude} and is given as the sum of Feynman integrals
\begin{equation}
\label{eq:FD}
\A^{(L)}_n(\p_i) =\sum_\Gamma \int_{\R^{DL}} \frac{N(\p_i,\bl_j)}{\prod_e (\q_e^2 - m_e^2) } \prod_{j=1}^L d^D \bl_j,
\end{equation}
where the summation is over certain labeled graphs $\Gamma$ called $L$-loop Feynman diagrams, depending on the choice of QFT, and the denominators involve linear combinations $\q_e$ of the external momenta $\p_i$ and the loop momenta $\bl_j$.  Here, $\q_e^2 := \q_e \cdot \q_e$, using the Lorentz inner product of $D$-dimensional space-time.  The numerator $N(\p_i,\bl_j)$ depends on the choice of QFT, the Feynman diagram $\Gamma$, and so on, and is a polynomial in the vectors $\p_i$ and $\bl_j$.  The product in the denominator is over the set of (internal) edges $e$ of the Feynman diagram $\Gamma$. 

The particles represented by the internal edges of $\Gamma$ are \emph{virtual particles}.  Real particles have momentum vectors satisfying the \emph{on-shell} condition: $\p^2 = m^2$.  Virtual particles can be \emph{off-shell}, and may not satisfy this equality.

When $L=0$, the function $\A^{(0)}_n(\p_i)$ is called the tree amplitude, and it is a rational function.  When $L > 0$, the analytic dependence of $\A^{(L)}_n(\p_i)$ on the $\p_i$ can be much more complicated.  Furthermore, the integrals in \eqref{eq:FD} do not in general converge.  Ultraviolet divergences occur due to the integration of large values of loop momentum space; this is dealt with via \emph{renormalization}.  Infrared divergences occur due to the presence of massless particles, leading to the vanishing of the denominator factors in \eqref{eq:FD}; this is typically dealt with via \emph{dimensional regularization}, expressing the integrals in \eqref{eq:FD} as functions of the space-time dimension $D$.  

A mathematically rigorous development of quantum field theory has long held the interest of mathematicians.  In particular, there have been substantial developments in the relation between Feynman integrals and the theory of periods and motives; see for example \cite{Mar}.  

\begin{example}\label{ex:phi3}
We discuss a concrete example that will make a number of appearances in this survey.  We consider massless scalar particles with cubic interaction, called $\phi^3$-theory.  In this theory there is one kind of particle, the Feynman diagrams are cubic graphs with leaves labeled $1,2,\ldots,n$, and the numerator factors $N(\p_i,\bl_j)$ are equal to 1.  For the tree-level ($L = 0$) amplitude the Feynman diagrams are trees and there are no integrals, and we have (up to a constant that depends on coupling constants, etc.)
\begin{equation}\label{eq:fullphi3}
\A_n^{(0)}(\p_i) = \sum_{\mbox{ {\mini cubic trees $T$ with $n$ leaves}}} \prod_{e \in I(T)} \frac{1}{X_e},
\end{equation}
where the product is over the internal edges $I(T)$ of $T$, and $X_e$ denotes the square of the momentum flowing along the edge $e$.  If the leaves labeled $i_1,i_2,\ldots,i_r$ are on one side of the edge $e$, then 
\begin{equation}\label{eq:Xe}
X_e = (\p_{i_1} + \p_{i_2} + \cdots + \p_{i_r})^2
\end{equation}
and by momentum conservation $\sum_{i=1}^n \p_i = 0$, this does not depend on which side of the edge $e$ we considered; see \cref{fig:Xe}.  Note that this is a rational function in the \emph{Mandelstam variables}, the Lorentz inner products of the momentum vectors, 
$$
s_{ij} := (\p_i + \p_j)^2 = 2 \p_i \cdot \p_j
$$
where we have used the massless condition $\p_i^2 = 0 = \p_j^2$.  Thus \eqref{eq:fullphi3} is a rational function on (massless) kinematic space $K_n$, the vector space with distinguished linear functions $s_{ij}$, $i \neq j$ satisfying
\begin{equation}
\label{eq:sij}
s_{ij} = s_{ji}, \qquad \sum_j s_{ij} = 0 \mbox{ for $i = 1,2,\ldots,n$},
\end{equation}
where the latter condition comes again from momentum conservation.  For $n = 4$ we would get
$$
\A_4^{(0)}(\p_i) = \frac{1}{(\p_1+\p_2)^2} + \frac{1}{(\p_1+\p_3)^2} + \frac{1}{(\p_1+\p_4)^2} = \frac{1}{s_{12}} + \frac{1}{s_{13}} + \frac{1}{s_{14}}.
$$
At $1$-loop, $\A_n^{(1)}$ is given by a sum of integrals of the form
$$
 \int_{\R^D} \frac{1}{\q^2 (\p_1- \q)^2(\p_1+\p_2- \q)^2 (\p_4 + \q)^2} d^D \q,
$$
the integral being over the $D$-dimensional momentum space $\R^D$ for $\q$.  This integral comes from the $1$-loop box diagram in \cref{fig:fourbox}.
\end{example}

\begin{figure}\mini
\begin{center}
$$
\begin{tikzpicture}
\coordinate (A) at (-0.5,0);
\coordinate (B) at (1,0);
\coordinate (C) at (2,-0.866);
\coordinate (D) at (3.5,-0.866);
\node (L1) at (-1,-0.866) {$1$};
\node (L2) at (-1,+0.866) {$2$};
\node (L3) at (1.5,+0.866) {$3$};
\node (L4) at (4,0) {$4$};
\node (L5) at (4,-2*0.866) {$5$};
\node (L6) at (1.5,-2*0.866) {$6$};
\draw[thick] (A) --(B)--(C)--(D)--(L4);
\draw[thick] (L1)--(A)--(L2);
\draw[thick] (L3)--(B);
\draw[thick] (D)--(L5);
\draw[thick] (C)--(L6);
\draw[thin,->,red] (2.8,1.3) to [bend left = 25] (1.5,-0.43);
\node at (0.25,0.2) {\mini $\pm(\p_1+\p_2)$};
\node at (2.8,1.5) {\mini $\pm(\p_1+\p_2+\p_3)$};
\node at (2.75,-1.05) {\mini $\pm(\p_4+\p_5)$};
\end{tikzpicture}
$$
\end{center}
\caption{A cubic planar tree with six leaves.  Each internal edge is labeled with the momentum of the particle traveling along it, which can be calculated by momentum conservation.}
\label{fig:Xe}
\end{figure}
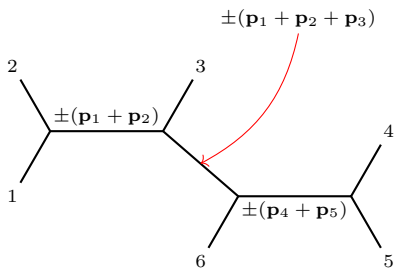

\section{Massless scattering amplitudes in gauge theory.}\label{sec:gauge}
Dixon \cite{Dix} calls scattering amplitudes ``the most perfect microscopic structures in the universe''.  Specifically, the scattering of massless particles in supersymmetric gauge theory, such as super Yang-Mills theory, exhibits surprising elegance and simplicity.  For physics textbook introductions, we refer the reader to \cite{EH,HePl}.  

A key feature of gauge theory is that one is led to the study of \emph{planar amplitudes}.  In gauge theories, particles have \emph{color}.  This is a vector in a chosen representation of the gauge group.  We take the gauge group to be $\SU(N)$, and consider the scattering of particles called gluons, including the ones appearing in quantum chromodynamics (QCD) describing the strong force.  The representation will be chosen to be the adjoint representation $\mathfrak{su}(N)$.  We fix a basis $T^1,T^2,\ldots,T^r$ for $\mathfrak{su}(N)$, and we say that a particle has color $a$, if it is labeled by the vector $T^a$.  Furthermore, we consider helicity amplitudes, so that every particle is assigned a helicity $+$ or $-$, roughly corresponding to the direction of angular momentum.  The tree-level amplitude then has the following decomposition:
\begin{equation}\label{eq:color}
\A_n^{\tree}({\p_i, h_i, a_i}) = C \sum_{\sigma \in S_n/Z_n} \Tr(T^{a_{\sigma(1)}} \cdots T^{a_{\sigma(n)}}) A_n^{\tree}(\sigma(1^h_n),\ldots,\sigma(n^{h_n})),
\end{equation}
where $a_i$ is the color of particle $i$, and $h_i$ is the helicity of particle $i$.  Here, the summation is over the set $S_n/Z_n$ of permutations of the particles, up to cyclic rotation.  The color-ordered amplitudes $A_n^{\tree}$ no longer depend on the color indices $a_i$ and is a summation over only the \emph{planar} Feynman tree diagrams.  Schematically, we have
$$
\text{full amplitude} = \sum \text{(group theory factor)} \cdot \text{(kinematic factor)} ,
$$
where the group theory factor depends on color indices and the gauge group, while the kinematic factor depends on the momentum vectors.  

At higher loops, the analogous formula to \eqref{eq:color} has more terms.  For example, at one-loop, there are two terms: one is a summation where the group theory factors are single traces; the other involves a product of two traces.  Sending $N \to \infty$ while keeping the 'tHooft coupling $\lambda= g^2 N$ fixed (here, $g$ is a coupling constant), the single trace term dominates, and we again have a summation over planar Feynman diagrams, the graphs having Betti number $L$.  In other words, in the large rank limit, the planar part of the amplitude dominates.  

While in the original setting \eqref{eq:FD}, the momenta are real $D$-dimensional vectors, we henceforth consider the amplitudes $A_n^{(L)}$ as complex analytic functions defined on complex momentum space.  In the 1980s, Parke and Taylor \cite{PT} discovered a remarkably simple formula for color-ordered massless gluon scattering amplitudes at tree-level:
\begin{equation}\label{eq:PT}
A_n^{\tree}(1^+,\ldots,i^-, \ldots, j^-,\ldots, n^+) = \frac {\langle i\;j\rangle ^{4}}{\langle 1\;2\rangle \langle 2\;3\rangle \cdots \langle (n-1)\;n\rangle \langle n\;1\rangle }.
\end{equation}
On the left hand side, the notation indicates that particles $i$ and $j$ have negative helicity, while the rest have positive helicity.  The right hand side is a rational function expressed in terms of \emph{spinors}.  In the $D=4$ spacetime dimensions of the real world, complexified momentum spacetime $\C^4$ can be identified with the space of $2 \times 2$ complex matrices $M$.  We choose the identification $\p \mapsto M(\p)$ so that $\p \cdot \p = \det(M(\p))$.  Since we assume the particles are massless, we obtain matrices $M(\p)$ of rank less than or equal to one.  Thus $M(\p)$ has a factorization as $M(\p)= \lambda \tla$, the product of a column 2-vector $\lambda$ with a row 2-vector $\tla$.  The vectors $\lambda$ and $\tla$ are called spinors, and the angle brackets $\langle i j \rangle$ is the $2 \times 2$ determinant of $[\lambda_i \lambda_j]$.  Conceptually, the complexified Lorentz group $\SO(4,\C)$ can be identified with $\SL(2,\C) \times \SL(2,\C)$ (modulo centers), which induces an isomorphism $\C^4 \cong \C^2 \otimes \C^2$.  Parke and Taylor's formula \eqref{eq:PT} is particularly remarkable because of the large number of Feynman diagrams that contribute to the computation.  For $n = 6$, this involves 220 Feynman diagrams.

For color-ordered amplitudes with more than two negative helicity gluons, the amplitude is more complicated.  Here is a six-gluon amplitude:
\begin{equation}\label{eq:NMHV}
A_6(1^-,2^-,3^-,4^+,5^+,6^+) = \frac{\langle 3|1+2|6]^3}{{\mathbf{P}}^2_{126}[21][16]\langle 34 \rangle \langle 45 \rangle \langle 5 |1 + 6|2]} +  \frac{\langle 1|5+6|4]^3}{{\mathbf{P}}^2_{156}[23][34]\langle 56 \rangle \langle 61 \rangle \langle 5 |1 + 6|2]},
\end{equation}
where the notation $\langle 3|1+2|6]$ is the product of a spinor column vector, a $2 \times 2$ matrix, and a spinor row vector, and ${\mathbf{P}}_{abc} = \p_a+\p_b+\p_c$.  In particular, when put over a common denominator, the numerator is a complicated polynomial.  Nevertheless, Parke and Taylor's discovery led Britto--Cachazo--Feng--Witten \cite{BCFW} to give a complete (BCFW) recursion for tree-level scattering amplitudes for massless gluons, a landmark result in the theory of scattering amplitudes.  

Could scattering amplitudes be determined by other principles, avoiding the elaborate summations \eqref{eq:FD} over Feynman diagrams?

\section{Schubert geometry.}\label{sec:Schubert}
We begin our journey into combinatorial geometry by identifying complexified momentum space $\C^4$ with the open Schubert cell in the Grassmannian $\Gr(2,4)$ of 2-planes in $\C^4$, or lines $L$ in $\P^3$, via the map
\begin{equation}\label{eq:G24}
\varphi: \p \longmapsto L = {\rm rowspan} \left(\begin{bmatrix} \Id_{2 \times 2} & M(\p) \end{bmatrix} \right) \in \Gr(2,4).
\end{equation}
The conformal group $\SO(4,2)$ is generated by the Lorentz group $\SO(3,1)$, translations, and inversions.  The complexified conformal group $\SO(6,\C)$ has the same Lie algebra as the special linear group $\SL(4,\C)$; this is the Dynkin diagram isomorphism $D_3 \cong A_3$.  The action of the conformal group extends to the natural action of $\SL(4,\C)$ on $\Gr(2,4)$.  We thus view $\Gr(2,4)$ as \emph{conformally compactified momentum space}.

\begin{proposition}\label{prop:G24}
Let $\y,\y’ \in \C^4$ be two momentum vectors, and $L,L' \in \Gr(2,4)$ be the images of $\y,\y’$ under \eqref{eq:G24}.  Then $(\y-\y’)^2 = 0$ if and only if the lines $L,L'$ intersect in $\P^3$.
\end{proposition}
\begin{proof}
The two lines $L,L'$ intersect if and only if the $4 \times 4$ matrix formed from them is singular.  We compute
\begin{align*}
\det \begin{bmatrix} \Id & M(\y) \\ \Id & M(\y’) \end{bmatrix} = \det(M(\y)-M(\y’)) = (\y-\y’)^2.
\end{align*}
\end{proof}

\Cref{prop:G24} puts the study of massless momentum vectors within the framework of classical Schubert calculus.   

\begin{definition}
Let $\p_1,\ldots,\p_n$ be nonzero massless momentum vectors satisfying momentum conservation $\p_1+ \cdots + \p_n = 0$.  The \emph{momentum twistors} $Z_1,Z_2,\ldots,Z_n \in \P^3$ of $\{\p_1,\ldots,\p_n\}$ are defined by 
$$
Z_i = \varphi(\y_i) \cap \varphi(\y_{i+1}),
$$
where $\y_1,\ldots,\y_n \in \MS$ are chosen so that $\p_i = \y_i - \y_{i+1}$.
\end{definition}
Since $\p_i \neq 0$, we have $\y_i \neq \y_{i+1}$, so $Z_i$ is well-defined by \cref{prop:G24}.  The $\y_i$ can be recovered from the $Z_i$ by $\varphi(\y_i) = \overline{Z_{i-1} Z_{i}}$.
The $\y_i$ are determined up to translation, which under $\varphi$ is part of the $\SL(4,\C)$-action on $\P^3$. 

To summarize, the ``external kinematic data'' may be viewed as a configuration of $n$ points, called momentum-twistors, in $\P^3$ up to the simultaneous action of $\SL(4,\C)$.  Amplitudes are thus functions on the Grassmannian $\Gr(4,n)$, or on the configuration space $\Conf(4,n)$ of $n$ points in $\P^3$.  Momentum-twistors were first introduced by Hodges \cite{Hod}, and they make manifest new symmetries of amplitudes.  Namely, the complexified conformal group $\SL(4,\C)$ acting on the $Z_i$-s is not the usual conformal group in position space. This conformal group is responsible for \emph{dual conformal symmetry}, part of a larger Yangian symmetry \cite{DHP} that planar super Yang-Mills amplitudes enjoy.  

A number of physical principles constrain the possible functions on configuration space that can occur as scattering amplitudes.  The principle of \emph{locality} states that poles of the amplitude are simple and only occur when the sum of some subset of the momenta go on-shell, that is, $(\sum_{i \in S} \p_i)^2 = 0$ for some set $S$.  This follows from the Feynman diagram description of amplitudes where the denominators involve squared momenta of the virtual particles.  The principle of \emph{unitarity} states that the scattering matrix is unitary.  This implies that at a pole of the amplitude one sees a factorization into smaller amplitudes.  In the case of color-ordered amplitudes, the locality condition is that $(\p_i + \p_{i+1} + \cdots + \p_{j-1})^2 = 0$ for a cyclic interval $[i,j]$.  If $(\p_i + \p_{i+1} + \cdots + \p_{j-1})^2 = 0$, then $(\y_i - \y_{j})^2=0$, so by \cref{prop:G24}, the two lines $\varphi(\y_i) = \overline{Z_{i-1}Z_i}$ and $\varphi(\y_{j}) = \overline{Z_{j-1}Z_j}$ intersect.  Thus locality singles out the configurations where 
\begin{equation}
\label{eq:locality}
\mbox{the four momentum-twistors $Z_{i-1},Z_i,Z_{j-1},Z_j$ lie on a plane in $\P^3$,}
\end{equation}
or equivalently, we have the vanishing of the $4 \times 4$ determinant $\langle Z_{i-1} Z_i Z_{j-1} Z_j \rangle$.  In more geometric language, we are interested in functions on configuration space $\Conf(4,n)$ that detects when certain four-tuples of points become coplanar.

Recall from \cref{sec:amplitudes} that the internal edges of a Feynman diagram correspond to virtual particles that are typically off-shell, that is, do not satisfy $\p^2 = m^2$.  The integrand of the Feynman integral acquires singularities whenever these particles go on-shell and become real particles.  Since we consider only massless particles, this condition is that the squared momenta $\p^2 = m^2 = 0$ vanishes.  This on-shell condition occurs in the study of the singularities of amplitudes, and by complex analysis (for instance, the residue theorem) in the study of the amplitudes themselves.  For example, putting the internal edges of the Feynman box diagram on-shell gives rise to the classical Schubert problem of finding lines $L$ that intersect four given lines $L_1,L_2,L_3,L_4$ in $\P^3$ (see \cref{fig:fourbox}).
\begin{figure}
\begin{center}
$$
\begin{tikzpicture}
\begin{scope}[decoration={
    markings,
    mark=at position 0.5 with {\arrow{>}}}
    ] 
\draw[decoration={markings, mark=at position 0.5 with {\arrow{>}}},postaction={decorate}] (0,0) -- (1,0);
\draw[decoration={markings, mark=at position 0.5 with {\arrow{>}}},postaction={decorate}]  (1,0) -- (1,1);
\draw[decoration={markings, mark=at position 0.5 with {\arrow{>}}},postaction={decorate}]  (0,1) -- (1,1);
\draw[decoration={markings, mark=at position 0.5 with {\arrow{>}}},postaction={decorate}]  (0,0)  -- (0,1);
\draw[decoration={markings, mark=at position 0.5 with {\arrow{>}}},postaction={decorate}]  (-0.5,-0.5)--(0,0);
\draw[decoration={markings, mark=at position 0.5 with {\arrow{>}}},postaction={decorate}]  (-0.5,1.5)--(0,1);
\draw[decoration={markings, mark=at position 0.5 with {\arrow{>}}},postaction={decorate}]  (1.5,-0.5)--(1,0);
\draw[decoration={markings, mark=at position 0.5 with {\arrow{>}}},postaction={decorate}]  (1.5,1.5)--(1,1);
\node at (-0.3,-0.7) {\mini $\p_1$};
\node at (-0.3,1.7) {\mini $\p_2$};
\node at (1.3,-0.7) {\mini $\p_4$};
\node at (1.3,1.7) {\mini $\p_3$};
\node at (0.5,-0.2) {\mini $\q$};
\node at (-0.5,0.5) {\mini $\p_1-\q$};
\node at (1.6,0.5) {\mini $\p_4+\q$};
\node at (0.5,1.2) {\tiny $\p_1 + \p_2 - \q$};
\begin{scope}[shift={(5,0)}]
\draw (0.5,0.5)--(2,0.5);
\draw (0.5,0.5)--(-1,0.5);
\draw (0.5,0.5)--(.5,2);
\draw (0.5,0.5)--(.5,-1);
\draw (2,0.5)--(.5,2)--(-1,0.5)--(.5,-1)--(2,0.5);
\draw[fill] ($(0.5,0.5)$) circle (0.05cm);
\draw[fill] ($(-1,0.5)$) circle (0.05cm);
\draw[fill] ($(0.5,2)$) circle (0.05cm);
\draw[fill] ($(0.5,-1)$) circle (0.05cm);
\draw[fill] ($(2,0.5)$) circle (0.05cm);
\node at (-1.2,0.5) {\mini $\y_2$};
\node at (0.5,-1.2) {\mini $\y_1$};
\node at (2.2,0.5) {\mini $\y_4$};
\node at (0.5,2.2) {\mini $\y_3$};
\node at (0.6,0.65) {\mini $\y_1 + \q$};
\end{scope}
\begin{scope}[shift={(10,0)}]
\draw[fill] ($(0.5,0.5)$) circle (0.05cm);
\draw[fill] ($(-1,0.5)$) circle (0.05cm);
\draw[fill] ($(0.5,2)$) circle (0.05cm);
\draw[fill] ($(0.5,-1)$) circle (0.05cm);
\draw[fill] ($(2,0.5)$) circle (0.05cm);
\draw (0.5,0.5)--(2,0.5);
\draw (0.5,0.5)--(-1,0.5);
\draw (0.5,0.5)--(.5,2);
\draw (0.5,0.5)--(.5,-1);
\draw (2,0.5)--(.5,2)--(-1,0.5)--(.5,-1)--(2,0.5);
\node at (-1.2,0.5) {\mini $L_2$};
\node at (0.5,-1.2) {\mini $L_1$};
\node at (2.2,0.5) {\mini $L_4$};
\node at (0.5,2.2) {\mini $L_3$};
\node at (0.6,0.65) {\mini $L$};
\node at (-0.4,-0.35) {\mini $Z_1$};
\node at (-0.4,1.4) {\mini $Z_2$};
\node at (1.4,1.4) {\mini $Z_3$};
\node at (1.4,-0.4) {\mini $Z_4$};
\node at (-0.2,0.35) {\mini $W_2$};
\node at (0.68,-0.2) {\mini $W_1$};
\node at (1.2,0.35) {\mini $W_4$};
\node at (0.68,1.2) {\mini $W_3$};
\end{scope}
\end{scope}
\end{tikzpicture}
$$
\end{center}
\caption{Left: the Feynman box diagram has four incoming particles with momenta $\p_1,\p_2,\p_3,\p_4$.  Center: drawing the dual planar graph, we assign dual momenta to the vertices so that the differences along edges are momenta in the left diagram.  Right: we put all the original momenta on-shell, and replace dual momenta with lines $L = \varphi(\y)$ labeling the vertices.  Edges in this dual graph represent intersecting lines and are labeled by the point of intersection.}
\label{fig:fourbox}
\end{figure}
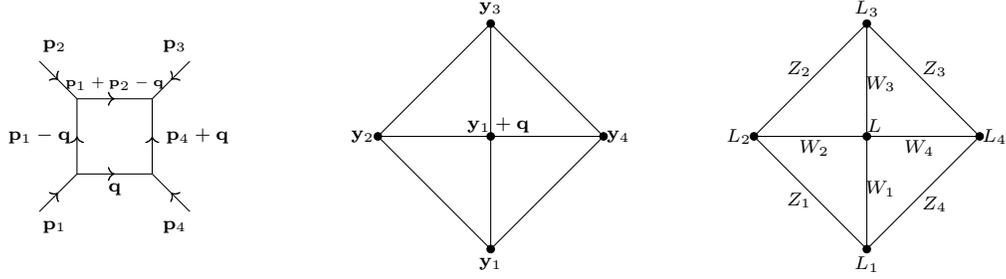

\section{Positroid geometry.}\label{sec:positroid}

When internal particles in a Feynman diagram satisfy the on-shell condition, we obtain graphs called \emph{on-shell diagrams}.  We will formalize the combinatorial geometry following the \emph{vector-relation configuration} conventions of \cite{AGPR}.  Instead of four-dimensions, we will work in arbitrary dimension.  Additional dimensions arize, roughly speaking, from (bosonization of) supersymmetry.

Let $G$ be a planar bipartite graph embedded into the disk, with boundary vertices labeled $1,2,\ldots,n$ in clockwise order.  The vertex set $B \sqcup W$ is partitioned into the black vertices $B$ and the white ones $W$.  Furthermore, we assume that boundary vertices belong to $W$ and are leaves.  Define $k = k(G) = |W| - |B|$.  We always assume that $G$ has an almost perfect matching: a collection of edges that uses all internal vertices, and (therefore) $n-k$ of the boundary vertices.

\begin{definition}
Let $G$ be a planar bipartite graph as above and $k = k(G)$.  A \emph{vector-relation configuration} $\WW$ for $G$ consists of 
\begin{itemize}
\item
a non-zero vector $W_w \in \C^k$ for each white vertex $w$, and
\item 
a non-trivial relation $\sum_i \alpha_{b,w_i} W_{w_i} = 0$ for each black vertex $b$, where $w_1,\ldots,w_r$ are the white vertices incident to $b$,
\end{itemize}
such that the matrix $\alpha_{b,w}$ is full-rank.
\end{definition}

To encode the vector-relation configuration, we may decorate the edge $(b,w)$ with the scalar $\alpha_{b,w}$.  From the boundary vertices, we obtain vectors $W_1,W_2,\ldots,W_n$, and we define
$$
V(\WW) = {\rm rowspan} \begin{bmatrix} | & | &  | & |  \\
W_1 & W_2 & \cdots & W_n \\
| & | &  | & |\end{bmatrix} \in \Gr(k,n).
$$

For a fixed $G$, the Zariski-closure of the set of subspaces $V(\WW)$ is a subvariety $\Pi_G$ of the Grassmannian.

\begin{definition}
A positroid variety $\Pi \subset \Gr(k,n)$ is a non-empty subvariety of the Grassmannian $\Gr(k,n)$ obtained by imposing rank conditions of the form 
$$
\rank(v_i,v_{i+1},\ldots,v_j) := \rank \begin{bmatrix} | & | &  | & |  \\
v_i & v_{i+1} & \cdots & v_{j} \\
| & | &  | & |\end{bmatrix}   \leq r_{ij}
$$
for various cyclic intervals $[i,j]$, and where $v_1,\ldots,v_n$ denote the columns of the $k \times n$ matrix representing a point in $\Gr(k,n)$.
\end{definition}

\begin{theorem}[{\cite{AGPR,LamCDM}}]
The subvariety $\Pi_G \subset \Gr(k,n)$ is a \emph{positroid variety}. 
\end{theorem}

Positroid varieties are irreducible subvarieties of the Grassmannian that can also be obtained as the intersection of $n$ cyclically rotated Schubert varieties.  The positroid stratification of the Grassmannian enjoys many favorable properties similar to the Schubert stratification: for instance, the open strata are irreducible and smooth, and the closed strata are normal and Cohen-Macaulay.  Positroid subvarieties of the Grassmannian can be indexed by \emph{bounded affine permutations}, or by \emph{Grassmann necklaces}, or by equivalence classes of \emph{plabic graphs} \cite{PosTP,KLS}. 

In summary, to each (planar) on-shell diagram $G$ we obtain a positroid subvariety $\Pi_G$ of the Grassmannian.  In turn, the geometry of the positroid variety produces functions as follows.  The momentum twistors $Z_1,Z_2,\ldots,Z_n$ determine a subGrassmannian $\Gr(k,Z) = \Gr(k,n-k-4) \subset \Gr(k,n)$.  One obtains a rational function of $Z_1,Z_2,\ldots,Z_n$, called the \emph{on-shell function}, by summing the evaluation of the canonical form $\Omega(\Pi_G)$ (see \cref{sec:positive}) of the positroid variety at the intersection points of $\Pi_G$ with $\Gr(k,Z)$; see \cite{Grassbook,BHmom}.  This brings the computation of scattering amplitudes into direct contact with the intersection theory of positroid varieties.  

Knutson, Speyer, and I showed that the cohomology class of a positroid variety $\Pi_f$ is represented by an affine Stanley symmetric function $\tF_f$ \cite{LamAS}.  The Schur function expansion of the symmetric function $\tF_f$ encodes all the intersection numbers between positroid varieties and Schubert varieties.  A direct combinatorial description of the intersection numbers of the most physical interest follows from the \emph{affine Pieri rule} of \cite{Lee, LLMS,Lamampl}.

\begin{figure}
\begin{center}
$$
\begin{tikzpicture}
\draw (0,0.6)--(0.5,-0)--(0,-0.6)--(-0.5,0)--(0,0.6);
\draw (0.5,0)--(0.9,0);
\draw (-0.5,0)--(-0.9,0);
\draw (0.9,0)--(22.5:1.2);
\draw (0.9,0)--(-22.5:1.2);
\draw (0,0.6)--(67.5:1.2);
\draw (0,0.6)--(112.5:1.2);
\draw (0,-0.6)--(-67.5:1.2);
\draw (0,-0.6)--(-112.5:1.2);
\draw (-0.9,0)--(157.5:1.2);
\draw (-0.9,0)--(-157.5:1.2);
\node at (157.5:1.35) {\mini $1$};
\node at (112.5:1.35) {\mini $2$};
\node at (67.5:1.35) {\mini $3$};
\node at (22.5:1.35) {\mini $4$};
\node at (-22.5:1.35) {\mini $5$};
\node at (-67.5:1.35) {\mini $6$};
\node at (-112.5:1.35) {\mini $7$};
\node at (-157.5:1.35) {\mini $8$};
\draw (0,0) circle (1.2cm);
\draw[fill=white] (0,0.6) circle (0.07cm);
\draw[fill=white] (0,-0.6) circle (0.07cm);
\draw[fill] ($(0.9,0)!0.5!(22.5:1.2)$) circle (0.07cm);
\draw[fill] ($(0.9,0)!0.5!(-22.5:1.2)$) circle (0.07cm);
\draw[fill] ($(0,0.6)!0.5!(67.5:1.2)$) circle (0.07cm);
\draw[fill] ($(0,0.6)!0.5!(112.5:1.2)$) circle (0.07cm);
\draw[fill] ($(0,-0.6)!0.5!(-67.5:1.2)$) circle (0.07cm);
\draw[fill] ($(0,-0.6)!0.5!(-112.5:1.2)$) circle (0.07cm);
\draw[fill] ($(-0.9,0)!0.5!(157.5:1.2)$) circle (0.07cm);
\draw[fill] ($(-0.9,0)!0.5!(-157.5:1.2)$) circle (0.07cm);
\draw[fill] (-0.5,0) circle (0.07cm);
\draw[fill] (0.5,0) circle (0.07cm);
\draw[fill=white] (-0.9,0.0) circle (0.07cm);
\draw[fill=white] (0.9,-0.0) circle (0.07cm);

\begin{scope}[shift={(5,0)}]
\draw (0,0.6)--(0.5,-0)--(0,-0.6)--(-0.5,0)--(0,0.6);
\draw (0.5,0)--(0.9,0);
\draw (-0.5,0)--(-0.9,0);
\draw (0.9,0)--(22.5:1.2);
\draw (0.9,0)--(-22.5:1.2);
\draw (0,0.6)--(67.5:1.2);
\draw (0,0.6)--(112.5:1.2);
\draw (0,-0.6)--(-67.5:1.2);
\draw (0,-0.6)--(-112.5:1.2);
\draw (-0.9,0)--(157.5:1.2);
\draw (-0.9,0)--(-157.5:1.2);
\node at (157.5:1.45) {\mini $W_1$};
\node at (112.5:1.45) {\mini $W_2$};
\node at (67.5:1.45) {\mini $W_3$};
\node at (22.5:1.45) {\mini $W_4$};
\node at (-22.5:1.45) {\mini $W_5$};
\node at (-67.5:1.45) {\mini $W_6$};
\node at (-112.5:1.45) {\mini $W_7$};
\node at (-157.5:1.45) {\mini $W_8$};
\draw (0,0) circle (1.2cm);
\draw[fill=white] (0,0.6) circle (0.07cm);
\draw[fill=white] (0,-0.6) circle (0.07cm);
\draw[fill] (-0.5,0) circle (0.07cm);
\draw[fill] (0.5,0) circle (0.07cm);
\draw[fill] ($(0.9,0)!0.5!(22.5:1.2)$) circle (0.07cm);
\draw[fill] ($(0.9,0)!0.5!(-22.5:1.2)$) circle (0.07cm);
\draw[fill] ($(0,0.6)!0.5!(67.5:1.2)$) circle (0.07cm);
\draw[fill] ($(0,0.6)!0.5!(112.5:1.2)$) circle (0.07cm);
\draw[fill] ($(0,-0.6)!0.5!(-67.5:1.2)$) circle (0.07cm);
\draw[fill] ($(0,-0.6)!0.5!(-112.5:1.2)$) circle (0.07cm);
\draw[fill] ($(-0.9,0)!0.5!(157.5:1.2)$) circle (0.07cm);
\draw[fill] ($(-0.9,0)!0.5!(-157.5:1.2)$) circle (0.07cm);
\node at (-0.65,-0.15) {\mini $W_*$};
\draw[fill=white] (-0.9,0.0) circle (0.07cm);
\draw[fill=white] (0.9,-0.0) circle (0.07cm);
\end{scope}

\begin{scope}[shift={(10,0)}]
\draw (0,0.6)--(0.5,-0)--(0,-0.6)--(-0.5,0)--(0,0.6);
\draw (0.5,0)--(0.9,0);
\draw (-0.5,0)--(-0.9,0);
\draw (0.9,0)--(22.5:1.2);
\draw (0.9,0)--(-22.5:1.2);
\draw (0,0.6)--(67.5:1.2);
\draw (0,0.6)--(112.5:1.2);
\draw (0,-0.6)--(-67.5:1.2);
\draw (0,-0.6)--(-112.5:1.2);
\draw (-0.9,0)--(157.5:1.2);
\draw (-0.9,0)--(-157.5:1.2);
\node at (157.5:1.35) {\mini $1$};
\node at (112.5:1.35) {\mini $2$};
\node at (67.5:1.35) {\mini $3$};
\node at (22.5:1.35) {\mini $4$};
\node at (-22.5:1.35) {\mini $5$};
\node at (-67.5:1.35) {\mini $6$};
\node at (-112.5:1.35) {\mini $7$};
\node at (-157.5:1.35) {\mini $8$};
\draw (0,0) circle (1.2cm);
\draw[fill=white] (0,0.6) circle (0.07cm);
\draw[fill=white] (0,-0.6) circle (0.07cm);
\draw[fill] ($(0.9,0)!0.5!(22.5:1.2)$) circle (0.07cm);
\draw[fill] ($(0.9,0)!0.5!(-22.5:1.2)$) circle (0.07cm);
\draw[fill] ($(0,0.6)!0.5!(67.5:1.2)$) circle (0.07cm);
\draw[fill] ($(0,0.6)!0.5!(112.5:1.2)$) circle (0.07cm);
\draw[fill] ($(0,-0.6)!0.5!(-67.5:1.2)$) circle (0.07cm);
\draw[fill] ($(0,-0.6)!0.5!(-112.5:1.2)$) circle (0.07cm);
\draw[fill] ($(-0.9,0)!0.5!(157.5:1.2)$) circle (0.07cm);
\draw[fill] ($(-0.9,0)!0.5!(-157.5:1.2)$) circle (0.07cm);
\draw[fill] (-0.5,0) circle (0.07cm);
\draw[fill] (0.5,0) circle (0.07cm);
\draw[fill=white] (-0.9,0.0) circle (0.07cm);
\draw[fill=white] (0.9,-0.0) circle (0.07cm);

\draw[very thick, color = red, -{Triangle[length=1.5mm, width=1.6mm]}] (157.5:1.2)--(-0.9,0);
\draw[very thick, color = red, -{Triangle[length=1.5mm, width=1.6mm]}] (-0.9,0)--(-0.5,0);
\draw[very thick, color = red, -{Triangle[length=1.5mm, width=1.6mm]}] (-0.5,0)--(0,-0.6);
\draw[very thick, color = red, -{Triangle[length=1.5mm, width=1.6mm]}] (0,-0.6)--(0.5,0);
\draw[very thick, color = red, -{Triangle[length=1.5mm, width=1.6mm]}] (0.5,0)--(0.9,0);
\draw[very thick, color = red, -{Triangle[length=1.5mm, width=1.6mm]}] (0.9,0)--(22.5:1.2);
\end{scope}
\end{tikzpicture}
$$
\end{center}
\caption{Left: an on-shell diagram, or plabic graph, where boundary vertices are assumed to be white.  Center: a vector-relation configuration consists of vectors $W_w$ assigned to each white vertex $w$, including the boundary vertices.  Right: the trip permutation sends $1$ to $4$ by turning left (right) at white (black) vertices. }
\label{fig:W}
\end{figure}
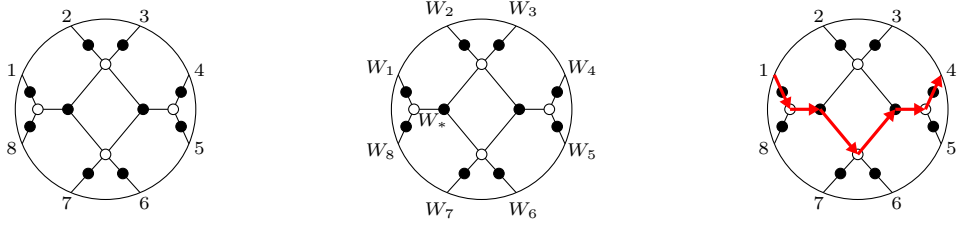

\begin{example}
Consider the on-shell diagram with $n = 8$ and $k = 2$, shown in \cref{fig:W}.  The $W$-s are 2-vectors.  The relation assumption says that $W_1$ and $W_*$ are parallel, and so are $W_8$ and $W_*$.  We conclude that $W_1$ and $W_8$ are parallel, and similarly, so are $\{W_2, W_3\}$, and $\{W_4, W_5\}$, and $\{W_6, W_7\}$.  These four rank conditions 
$$
\rank(v_1,v_8) \leq 1, \qquad \rank(v_2,v_3) \leq 1, \qquad \rank(v_4,v_5) \leq 1, \qquad \rank(v_6,v_7) \leq 1, 
$$
define the positroid variety $\Pi_G$ for this graph $G$.

Using Postnikov's ``rules of the road’’ (turn left at white and turn right at black), we may compute that the bounded affine permutation of this graph is $f = [4,3,6,5,8,7,10,9]$ in one-line window notation.  That is $f(1) = 4$, $f(2) = 3$, and so on.  The affine Stanley symmetric function is equal to $\tF_f = s_1^4$, where $s_1$ is the Schur function.  The coefficient of $s_{22}$ in the expansion of $s_1^4$ is equal to $2$, which is the intersection number relevant for the computation of the on-shell function.  See \cite{Lamampl} for an explanation of this computation.
\end{example}

The on-shell functions associated to these diagrams are important building blocks for amplitudes.  Indeed, the BCFW recursion for tree-level amplitudes can be made completely graphical using on-shell diagrams.  More generally, at loop-level one obtains the integrand using on-shell diagrams, and the formulation is most natural in the maximal supersymmetric extension of gauge theory: $N=4$ super Yang-Mills (SYM) theory.  On the one hand, supersymmetry is an elegant way to keep track of the different ``helicity sectors'' of gluon scattering, that is, it is a generating function for the different positions of $+$'s and $-$'s in $A_n(1^+,2^-,3^-,\ldots,n^+)$.  On the other hand, the $N=4$ SYM theory postulates a multiplet of particles: in addition to the gluons discussed in \cref{sec:gauge}, one has four fermions (spin $1/2$ particles) and six scalars (spin $0$ particles). 

\begin{example}
We consider the BCFW recursion for $k=1$.  We are thus working inside the Grassmannians $\Gr(1,n) = \P^{n-1}$, and the positroid varieties are coordinate hypersurfaces.  In this case, the plabic graphs consist of a single white interior vertex connected to $r$ of the boundary vertices (with degree two black vertices added since our convention is that boundary vertices are white), and we have $r = 5$ for the ones relevant for the amplitude.  Let $[a,b,c,d,e]$ denote the on-shell function corresponding to the graph where a single white interior vertex is connected to boundary vertices with labels $a,b,c,d,e$.  The BCFW recursion for $N=4$ SYM tree amplitudes in these coordinates is:
$$
A_{n,k=1} = A_{n-1,k=1}+ \sum_{j=3}^{n-2}[1,j-1,j,n-1,n],
$$
which can be solved to give
\begin{equation}\label{eq:k1}
A_{n,k=1} = \sum_{i < j}[1, i-1, i, j-1,j].
\end{equation}
(There is an overall ``MHV'' factor that appears when changing to momentum-twistor coordinates.  We do not include it in this expression.)
By cyclic symmetry, we also obtain other possible expressions for the amplitude.  For example, with $n = 6$, we get two possible expressions
\begin{equation}\label{eq:2triang}
A_{n,k=1} = [1,2,3,4,5]+[1,2,3,5,6]+[1,3,4,5,6] = [2,3,4,5,6]+[1,2,4,5,6]+[1,2,3,4,6].
\end{equation}
\end{example}

We have not discussed the Feynman diagrams of Yang-Mills theory.  For a more comprehensive discussion of how to relate the formulae \eqref{eq:2triang} directly to the original Feynman diagrams of Yang-Mills theory, we refer the reader to the review article \cite{DPSV}.

\section{Amplituhedron geometry.}\label{sec:ampli}
Hodges \cite{Hod} proposed that \eqref{eq:2triang} should be viewed as two different triangulations of some geometry.  This suggestion was supported by the fact that the rational functions $[a,b,c,d,e]$ had ``spurious poles” that were not present in the amplitude, or similarly the pole $\langle 5 |1 + 6|2]$ that turns out to cancel out in \eqref{eq:NMHV}.  These poles would correspond to internal boundaries of the triangulation that were not boundaries of the geometry itself.  Arkani-Hamed and Trnka \cite{AT} constructed the \emph{amplituhedron} as a geometric representation of the planar SYM amplitude.

Let $\Gr(k,n)_{\geq 0}$ denote the totally nonnegative Grassmannian \cite{PosTP,LusTP}, the locus inside the real Grassmannian where all Pl\"ucker coordinates are nonnegative.  Let $Z: \R^n \to \R^{k+m}$ be a linear map represented by a real matrix.  We call $Z$ \emph{positive} if it has positive maximal ($k+m \times k+m$) minors.  The linear map $Z$ induces a rational map 
$$
Z:\Gr(k,n) \dashrightarrow \Gr(k,k+m), \qquad V \mapsto Z(V).
$$

\begin{definition}
Let $Z:\R^n \to \R^{k+m}$ be positive.  The \emph{amplituhedron} is the image $A_{n,k,m} := Z(\Gr(k,n)_{\geq 0}) \subset \Gr(k,k+m)$ of the totally nonnegative Grassmannian. \end{definition}

The amplituhedron is a compact subspace of the real Grassmannian $\Gr(k,k+m)(\R)$.  Perhaps the most potent way to think of the amplituhedron is as a Grassmannian polytope.  When $k = 1$, the space $\Gr(1,n)_{\geq 0}$ is simplex $\Delta^{n-1}$ in projective space $\P^{n-1}$.  The image of $\Delta^{n-1}$ under a linear map $Z: \P^{n-1} \to \P^{d-1}$ is a polytope.  We call the image of $\Gr(k,n)_{\geq 0}$ under an arbitrary linear map a \emph{Grassmann polytope} \cite{LamCDM}.  For a positive $Z$, one obtains a cyclic polytope, and thus, the amplituhedron is a Grassmannian analogue of the cyclic polytope.

The central questions about Grassmann polytopes concern their combinatorial structure (face poset) and homeomorphism type as a stratified space.  In \cite{GKL1}, we showed with Galashin and Karp that $\Gr(k,n)_{\geq 0}$, the ``Grassmannian simplex'', is homeomorphic to a closed ball.  The intersection of $\Gr(k,n)_{\geq 0}$ with the stratification of the Grassmannian by positroid varieties endows $\Gr(k,n)_{\geq 0}$ with the structure of a regular CW-complex \cite{GKL3} whose open cells are called \emph{positroid cells} \cite{PosTP}.

It is conjectured that the amplituhedron is a ball for any positive $Z$.  This is known \cite{GKL1, BlKa} for some choices of $Z$, including the ``cyclically symmetric'' one.  The facet structure of the amplituhedron is not known except in polytope case or the $m =2$ case \cite{Lamm=2}, and even the question of the definition of the faces is not completely settled.  However, among the boundaries of the $m = 4$ amplituhedron are the hypersurfaces $\langle Y Z_{i-1} Z_i Z_{j-1} Z_j \rangle = 0$.  Here, $Y \in \Gr(k,k+4)$ is represented by a $k \times k+4$ matrix, the $Z$-s\footnote{We have $4$ spacetime dimensions and $k$ bosonized supersymmetry dimensions.} are vectors in $\R^{k+4}$, and $\langle Y Z_{i-1} Z_i Z_{j-1} Z_j \rangle$ denotes a $(k+4) \times (k+4)$ determinant.  This matches with the location of poles of the amplitude imposed by locality \eqref{eq:locality}, and was one of the ``smoking guns” pointing towards the amplituhedron.  Readers familiar with polytopal geometry will also recognize the appearance of Gale's evenness condition for the faces of the cyclic polytope.

``Triangulations” of the amplituhedron are certain collections of positroid cells whose images cover the amplituhedron.  In \cite{AT}, it is proposed that for $m = 4$, triangulations of the amplituhedron would give rise to collection of positroid varieties that give a formula for the amplitude: simply replace each positroid variety by the corresponding on-shell function and take the sum.  Indeed, in the case $k=1$, the amplituhedron is a cyclic polytope and \eqref{eq:k1} represents a triangulation of the cyclic polytope.  Significant advances in the study of triangulations of the $m=4$ amplituhedron have been made over the last decade; see \cite{GLparity,ELPSTW}.

More generally, understanding the combinatorics and geometry of Grassmann polytopes is a rich area for future exploration.

Let us emphasize the essential role that planarity plays in this story.  In physics, the study of color-ordered amplitudes leads to the planarity condition on Feynman diagrams and on-shell diagrams.  In combinatorics, the study of total positivity is intertwined with the theory of planar graphs in the classical works of Lindstr\"om and Gessel-Viennot, and subsequently in Postnikov's generalization to the totally nonnegative Grassmannian.  Planarity and positivity go hand-in-hand.

\section{Positive geometry.}\label{sec:positive}
It has long been a dream in physics that scattering amplitudes could be characterized as the unique complex analytic function satisfying certain properties.  As we have already mentioned, \emph{locality} characterizes the location of poles of the amplitude, and \emph{unitarity} further constrains the behavior of the amplitude near these singularities.  Could one find a set of such conditions that completely determine the amplitude?  With Arkani-Hamed and Bai \cite{ABL}, we introduced \emph{positive geometries} to mathematically formulate such a phenomenon.  Positive geometries are a mechanism that produces functions (such as the amplitude) from a geometry (such as the amplituhedron), without relying on any particular triangulation.

We first recall the definition of the residue of a meromorphic form.  Let $X$ be a complex $d$-dimensional irreducible, normal algebraic variety.  Let $H \subset X$ an (irreducible) hypersurface, and $\omega$ be a meromorphic $d$-form on $X$ with at most simple poles on $H$.  Let $f$ be a local coordinate such that $f$ vanishes to order one on $H$.  Write 
$$
\omega = \eta \wedge \frac{df}{f} + \eta' 
$$
for a $(d-1)$-form $\eta$ and a $d$-form $\eta'$, both without poles along $H$.  Then, the \emph{residue} $\Res_H \omega$ is the $(d-1)$-form on $H$ defined as the restriction
\begin{equation}\label{eq:resdef}
\Res_H \omega:= \eta|_H,
\end{equation}
and does not depend on the choices of $f, \eta,\eta'$.

\medskip
We now further assume that $X$ is a variety defined over $\R$ and equip the real points $X(\R)$ with the analytic topology.
Let $X_{\geq 0} \subset X(\R)$ be a closed semialgebraic subset such that the interior $X_{>0} = \Int(X_{\geq 0})$ is an oriented real $d$-manifold, and the closure of $X_{>0}$ recovers $X_{\geq 0}$.
Let $\partial X_{\geq 0}$ denote the boundary $X_{\geq 0} \setminus X_{ > 0}$ and let $\partial X$ denote the Zariski closure of $\partial X_{\geq 0}$.  Let $C_1,C_2,\ldots,C_r$ be the codimension 1 irreducible components of $\partial X$.  
Let $C_{i, \geq 0}$ denote the analytic closures of the interior of $C_i \cap \partial X_{\geq 0}$ in $C_i(\R)$.
The spaces $C_{1,\geq 0}, C_{2,\geq 0},\ldots, C_{r,\geq 0}$ are called the boundary components, or facets, of $X_{\geq 0}$.

\begin{definition}\label{def:PG}
  We call $(X,X_{\geq 0})$ a \emph{positive geometry} if there exists a unique nonzero rational $d$-form $\Omega(X,X_{\geq 0})$, called the \emph{canonical form}, satisfying the recursive axioms:
  \begin{enumerate}
    \item If $d = 0$, then $X = X_{\geq 0}= \pt$ is a point and we define $\Omega(X,X_{\geq 0})=\pm 1$ depending on the orientation.
    \item If $d>0$, then we require that $\Omega(X,X_{\geq 0})$ has poles only along the boundary components $C_i$, these poles are simple, and for each $i =1,2,\ldots,r$, we have
    \begin{equation}\label{eq:PGdef}
    \Res_{C_i}\Omega(X,X_{\geq 0})=\Omega(C_i,C_{i,\geq0}).
    \end{equation}
  \end{enumerate}
\end{definition}

\begin{example}
The interval $[a,b] \subset \R \subset \P^1(\R)$ sitting inside real projective space is an example of a positive geometry.  Its canonical form is 
\begin{equation}\label{eq:int}\Omega([a,b]) =\left( \frac{1}{x-a}  -\frac{1}{x-b}\right) dx,
\end{equation}
with simple poles at $x = a$ and $x = b$.  The residues at these two poles are opposite, and this matches with the fact that the boundary points $a$ and $b$ of $[a,b]$ have opposite orientations.  The fact that the residues are opposite guarantees that there is no pole at infinity.  The form $\Omega$ is the unique rational 1-form on $\P^1$ with these residues.
\end{example}

The central motivating conjecture in positive geometry is the following.
\begin{conjecture}
The amplituhedron $A_{n,k,m}$ is a positive geometry, and the canonical form $\Omega(A_{n,k,4})$ ``is'' the $N=4$ SYM amplitude.
\end{conjecture}
The statement that $A_{n,k,m}$ is a positive geometry is known in some cases: for $k = 1$, the amplituhedron $A_{n,1,m}$ is a polytope, so it follows from \cref{thm:dual}; for $m = 1$, it follows from \cite{KW}; for $k = m = 2$, it follows from \cite{RST}; for $n = k+m$ the amplituhedron $A_{k+m,k,m}$ is the totally nonnegative Grassmannian $\Gr(k,n)_{\geq 0}$ and it follows from \cite{ABL,LamPosGeom}.

There is a long and growing list of spaces that have been proven to be positive geometries: polytopes, positive parts of toric varieties, totally positive parts of Grassmannians and flag varieties, the positive part of $\M_{0,n}$, and so on; see \cite{LamPosGeom}.  The recent volume \cite{PGvolume} gives a snapshot of some recent advances in the field.

\section{Polytopal geometry.}\label{sec:polytope}
Polytopes are the most familiar positive geometries.  The survey \cite{LamPosGeom} introduces positive geometries from this perspective.  Among the many formulae for the canonical form of a polytope, we have the dual volume formula.

\begin{theorem}\label{thm:dual}
Let $P \subset \R^d \subset \P^d(\R)$ be a full-dimensional convex polytope.  Then $P$ is a positive geometry with canonical form
$$
\Omega(P) = \Vol((P- \x)^\vee) dx_1 dx_2 \cdots dx_d.
$$
\end{theorem}
Here, $P-\x$ is a translation of $P$ and $(P-\x)^\vee$ denotes the polar polytope.  We call the function $\Vol((P- \x)^\vee)$ the \emph{dual volume function}; see \cite{GLX} for a discussion in the context of (usual) convex geometry.  For $\x$ in the interior of $P$, this is well-defined and analytically continues to a rational function.  It is immediate that when $\x$ approaches the boundary of $P$ that this function acquires a singularity.  This is a direct analogue of the locality of scattering amplitudes.  As an example, the interval $[a,b] \subset \R$ has canonical form
$$
\Omega([a,b]) = \Vol\left(\left[\frac{1}{a-x},\frac{1}{b-x}\right]\right) dx =  \left( \frac{1}{x-a}  -\frac{1}{x-b}\right) dx,$$
agreeing with \eqref{eq:int}.

The positive geometry of polytopes appears in a number of places in physics, even beyond amplitudes, such as in cosmology.  In the following, we connect to the $\phi^3$-theory of \cref{ex:phi3}.

We consider the biadjoint cubic scalar theory.  This theory has only one type of particle, and cubic interactions, so that the Feynman diagrams in the theory are graphs where all internal vertices have degree 3.  The color-ordered amplitudes of this theory are the same as for the $\phi^3$-theory of \cref{ex:phi3} except that we are summing over planar diagrams.  We denote the set of cubic planar trees by $\T^{(3)}_n$.  Denoting the cubic planar scalar amplitude by $A_n^{\phi^3}$, we have
\begin{equation}\label{eq:planarphi3}
 A_{n}^{\phi^3} = \sum_{T \in \T^{(3)}_{n}} \prod_{e \in I(T)} \frac{1}{X_e},
\end{equation}
where $X_e$ is defined in \eqref{eq:Xe}.  The function $A_n^{\phi^3}$ should be viewed as a rational function on kinematic space $K_n$.  For example, for $n = 5$, we have
\begin{equation}\label{eq:5phi3}
A_n^{\phi^3} =  \frac{1}{s_{12}s_{45}} +  \frac{1}{s_{12}s_{34}} +  \frac{1}{s_{34}s_{15}}+  \frac{1}{s_{23}s_{15}} +  \frac{1}{s_{23}s_{45}}.
\end{equation}

Let $\Ass_{n-3}$ denote the $(n-3)$-dimensional associahedron.  This is a simple convex polytope with face poset given by the poset of planar trees $\T_n$ with $n$ leaves.  These are planarly embedded trees $T$ with $n$ boundary leaves labeled $1,2,\ldots,n$ in cyclic order.  The covering relation of $\T_n$ is as follows: for two trees $T,T'$, we have $T \lessdot T'$ if $T'$ is obtained from by contracting an internal edge.  The minimal elements of the poset $\T_n$ consist of the cubic planar trees.

For a general simple polytope $P \subset \R^d$, the canonical form $\Omega(P)$ has an expression as the sum over vertices
\begin{equation}
\label{eq:simple}
\Omega(P) = \left(\sum_{V(P)} \pm \prod_F \frac{1}{F(x)} \right) dx_1 \cdots dx_d
\end{equation}
where the product is over the $d$ facets incident to the vertex $v$, and $F(x)$ is the linear function corresponding to the facet $F$, scaled appropriately; see \cite{GLX}.  Since the vertex set of $\Ass_{n-3}$ can be identified with $\T^{(3)}_n$, the canonical form of the associahedron has an expression as a summation over cubic planar trees.  In \cite{ABHY}, a particular realization of $\Ass_{n-3}$ inside kinematic space $K_n$ is given so that the canonical form 
$\Omega(\Ass_{n-3})$ is equal to the rational function $A_n^{\phi^3}$ up to a standard top-form.  Thus the associahedron is the amplituhedron for biadjoint cubic scalar theory.    

We state a version of this result using the \emph{dual mixed volume} function in our joint work \cite{GLX} with Gao and Xue.

\begin{definition}
For subsets $P_1,\ldots,P_r$ in $\R^d$, the dual mixed volume is the function
$$
f(\lambda_1,\ldots,\lambda_r):= \Vol((\lambda_1 P_1 + \lambda_2 P_2 + \cdots +\lambda_r P_r)^\vee)
$$
where $P^\vee$ denotes the polar body and $P+Q$ denotes Minkowski sum.  
\end{definition}
When the $P_i$ are polytopes, this is a rational function in the $\lambda_i$ \cite{GLX}.  We choose the polytopes to be the simplices 
$$
\Delta_{[a,b]} := \conv(\epsilon_i \mid a \leq i \leq b) \subset \R^n,
$$
and use as our variables the Mandelstam variables $s_{a,b}$.
For positive $s_{a,b}$, the Minkowski sum $\Sigma_{1 \leq a \leq b \leq n -2} s_{a,b+1} \Delta_{[a,b]}$ is the \emph{Loday realization} of the associahedron $\Ass_{n-3}$.  We consider the dual mixed volume associated to this Minkowski sum; see \cref{fig:Ass}.

\begin{theorem}[\cite{GLX}]\label{thm:GLX}
The dual mixed volume of the associahedron $\Sigma_{1 \leq a \leq b \leq n-2} s_{a,b+1} \Delta_{[a,b]}$ is, up to sign, equal to the planar $\phi^3$-amplitude.
\end{theorem}

\begin{figure}
\begin{center}
$$
\begin{tikzpicture}

\draw[thick,color = blue] (0,0) -- (3,0) -- (3,1.5) -- (1.5,3) -- (0,3) -- (0,0);
\draw[fill,color=blue] (0,0) circle (0.5mm);
\draw[fill,color=blue] (3,0) circle (0.5mm);
\draw[fill,color=blue] (3,1.5) circle (0.5mm);
\draw[fill,color=blue] (1.5,3) circle (0.5mm);
\draw[fill,color=blue] (0,3) circle (0.5mm);
\begin{scope}[shift={(-0.5,3.4)},scale=0.3]
\coordinate (A) at (0,0);
\coordinate (B) at (1,0);
\coordinate (C) at (1.5,-0.866);
\node[inner sep = 0] (L1) at (-0.5,-0.866) {$1$};
\node[inner sep = 0] (L2) at (-0.5,+0.866) {$2$};
\node[inner sep = 0] (L3) at (1.5,+0.866) {$3$};
\node[inner sep = 0] (L4) at (2.5,-0.866) {$4$};
\node[inner sep = 0] (L5) at (1,-2*0.866) {$5$};
\draw[thick] (A) --(B)--(C)--(L4);
\draw[thick] (L1)--(A)--(L2);
\draw[thick] (L3)--(B);
\draw[thick] (C)--(L5);
\end{scope}

\begin{scope}[shift={(-0.45,-0.4)},scale = 0.3]
\coordinate (A) at (0,0);
\coordinate (B) at (1,0);
\coordinate (C) at (1.5,+0.866);
\node[inner sep =0] (L1) at (-0.5,-0.866) {$1$};
\node[inner sep =0](L2) at (-0.5,+0.866) {$2$};
\node[inner sep =0](L3) at (1,1.732) {$3$};
\node[inner sep =0](L4) at (2.5,+0.866) {$4$};
\node[inner sep =0] (L5) at (1.5,-0.866) {$5$};
\draw[thick] (A) --(B)--(C)--(L4);
\draw[thick] (L1)--(A)--(L2);
\draw[thick] (L3)--(C);
\draw[thick] (B)--(L5);
\end{scope}

\begin{scope}[shift={(3.3,-0.2)}, scale = 0.3]
\coordinate (A) at (0,0);
\coordinate (B) at (1,0);
\coordinate (C) at (1.5,-0.866);
\node[inner sep =0] (L5) at (-0.5,-0.866) {$5$};
\node[inner sep =0] (L1) at (-0.5,+0.866) {$1$};
\node[inner sep =0] (L2) at (1.5,+0.866) {$2$};
\node[inner sep =0] (L3) at (2.5,-0.866) {$3$};
\node[inner sep =0] (L4) at (1,-2*0.866) {$4$};
\draw[thick] (A) --(B)--(C)--(L3);
\draw[thick] (L5)--(A)--(L1);
\draw[thick] (L2)--(B);
\draw[thick] (C)--(L4);
\end{scope}

\begin{scope}[shift={(3.3,1.5)},scale = 0.3]
\coordinate (A) at (0,0);
\coordinate (B) at (1,0);
\coordinate (C) at (1.5,+0.866);
\node[inner sep =0] (L5) at (-0.5,-0.866) {$5$};
\node[inner sep =0] (L1) at (-0.5,+0.866) {$1$};
\node[inner sep =0] (L2) at (1,1.732) {$2$};
\node[inner sep =0] (L3) at (2.5,+0.866) {$3$};
\node[inner sep =0] (L4) at (1.5,-0.866) {$4$};
\draw[thick] (A) --(B)--(C)--(L3);
\draw[thick] (L5)--(A)--(L1);
\draw[thick] (L2)--(C);
\draw[thick] (B)--(L4);
\end{scope}

\begin{scope}[shift={(1.5,3.45)},scale = 0.3]
\coordinate (A) at (0,0);
\coordinate (B) at (1,0);
\coordinate (C) at (1.5,-0.866);
\node[inner sep =0] (L4) at (-0.5,-0.866) {$4$};
\node[inner sep =0] (L5) at (-0.5,+0.866) {$5$};
\node[inner sep =0] (L1) at (1.5,+0.866) {$1$};
\node[inner sep =0] (L2) at (2.5,-0.866) {$2$};
\node[inner sep =0] (L3) at (1,-2*0.866) {$3$};
\draw[thick] (A) --(B)--(C)--(L2);
\draw[thick] (L5)--(A)--(L4);
\draw[thick] (L1)--(B);
\draw[thick] (C)--(L3);
\end{scope}

\begin{scope}[shift={(8.4 ,0)}]
\draw[thick,color = blue] (0,0)--(1.5,0)--(0,1.5)--(0,0);
\begin{scope}[shift={(0,2.5)}]
\draw[thick,color = blue] (0,0)--(1.5,0);
\node at (-0.15,0){$\epsilon_2$};
\node at (1.65,0){$\epsilon_1$};
\end{scope}
\begin{scope}[shift={(-1,0)}]
\draw[thick,color = blue] (0,0)--(0,1.5);
\node at (0,-0.15){$\epsilon_2$};
\node at (0,1.65){$\epsilon_3$};
\end{scope}
\node at (-0.5,0.75) {$+$};
\node at (0.75,2) {$+$};
\node at (2.3,1){$=$};
\node at (-0.15,-0.15){$\epsilon_2$};
\node at (1.65,-0.15){$\epsilon_1$};
\node at (-0.15,1.65){$\epsilon_3$};
\end{scope}

\begin{scope}[shift={(11.5,0)}]
\draw[thick,color = blue] (0,0) -- (3,0) -- (3,1.5) -- (1.5,3) -- (0,3) -- (0,0);
\draw[fill,color=blue] (0,0) circle (0.5mm);
\draw[fill,color=blue] (3,0) circle (0.5mm);
\draw[fill,color=blue] (3,1.5) circle (0.5mm);
\draw[fill,color=blue] (1.5,3) circle (0.5mm);
\draw[fill,color=blue] (0,3) circle (0.5mm);
\node at (-0.1,-0.2) {\mini $(s_{12},s_{34})$};
\node at (-0.1,3.2) {\mini $(s_{12},s_{14}+s_{24}+s_{34})$};
\node at (3.2,2.8) {\mini $(s_{12}+s_{13},s_{14}+s_{24}+s_{34})$};
\node at (3.2,1.3) {\mini $(s_{12}+s_{13}+s_{14},s_{24}+s_{34})$};
\node at (3,-0.2) {\mini $(s_{12}+s_{13}+s_{14},s_{34})$};
\end{scope}

\end{tikzpicture}
$$
\end{center}
\caption{Left: The $2$-dimensional associahedron $\Ass_2$ is a pentagon.  The vertices have been labeled  by the five cubic planar trees with five leaves.  Right: The associahedron in \cref{thm:GLX} is a Minkowski sum (we do not depict the $0$-simplices which give a translation of the polytope).  We have projected the picture to the plane by omitting the second coordinate (corresponding to the basis vector $\epsilon_2$).}
\label{fig:Ass}
\end{figure}
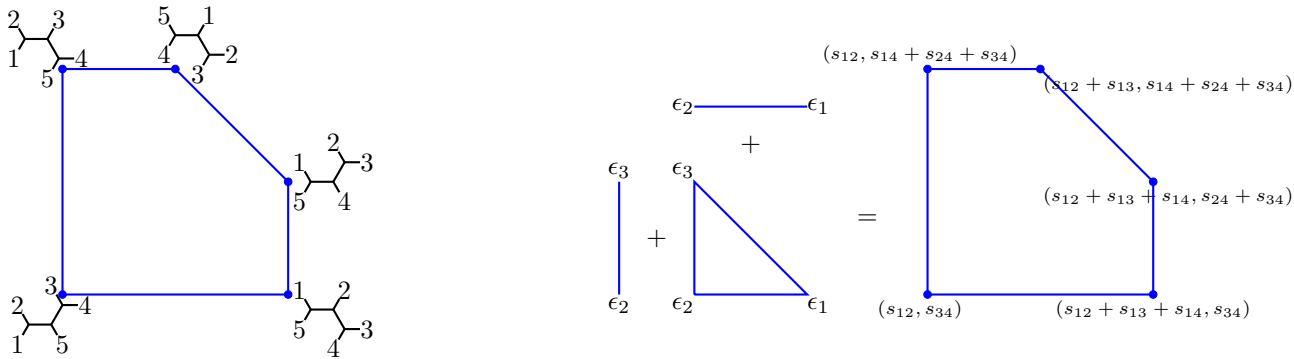

This story extends to the setting of generalized associahedra \cite{BDMTY,AHLcluster} appearing in the theory of cluster algebras.

\section{Moduli geometry.}\label{sec:moduli}
From the quantum field theory perspective, the cubic scalar theory is a toy model.  However, it arises as the low energy limit of string theory.

Whereas $n$ particle quantum field theory amplitudes are perturbative sums over Feynman diagrams with $n$ leaves and Betti number $L = 0,1,2,\ldots$, string theory amplitudes are integrals over the moduli space $\M_{g,n}$ of $n$-pointed genus $g = 0,1,2,\ldots$ curves.  The tree-level string amplitude is an integral over the moduli space $\M_{0,n}$ of $n$-points on $\P^1$.   This is a smooth affine algebraic variety of dimension $n - \dim \PGL(2) = n - 3$.

Roughly speaking, the \emph{closed} string amplitude is an integral over $\M_{g,n}(\C)$, while the \emph{open} string amplitude is an integral over $\M_{g,n}(\R)$.  \Cref{fig:string} depicts how closed strings (the circle $S^1$) or open strings (the interval $[0,1]$) scatter.  When the string length goes to 0, strings turn into particles, string theory turns into quantum field theory, and the picture of strings interacting turn into Feynman diagrams.
\begin{figure}
\begin{center}
$$
\begin{tikzpicture}
\draw plot [smooth,tension = 1] coordinates {(0,0.2)(2,-0.3)(4,0.2)};
\draw plot [smooth,tension = 1] coordinates {(0,-0.2)(1.1,-1)(0,-2)};
\draw plot [smooth,tension = 1] coordinates {(0,-2.4)(2,-1.9)(4.5,-3.4)};
\draw plot [smooth,tension = 1] coordinates {(4,-0.2)(2.7,-1)(4.5,-3)};
\draw (0,0) ellipse (0.1 and 0.2);
\draw (0,-2.2) ellipse (0.1 and 0.2);
\draw (4,0) ellipse (0.1 and 0.2);
\draw (4.5,-3.2) ellipse (0.1 and 0.2);
\begin{scope}[shift={(0.2,0.2)}]
\draw plot [smooth,tension = 1] coordinates {(1,-1)(1.15,-1.06)(1.4,-1.1)(1.65,-1.06)(1.8,-1)};
\draw plot [smooth,tension = 1] coordinates {(1.15,-1.06)(1.4,-1)(1.65,-1.06)};
\end{scope}
\begin{scope}[shift={(0.7,-0.2)}]
\draw plot [smooth,tension = 1] coordinates {(1,-1)(1.15,-1.06)(1.4,-1.1)(1.65,-1.06)(1.8,-1)};
\draw plot [smooth,tension = 1] coordinates {(1.15,-1.06)(1.4,-1)(1.65,-1.06)};
\end{scope}

\begin{scope}[shift={(6,0)}]
\draw plot [smooth,tension = 1] coordinates {(0,0.2)(2,-0.3)(4,0.2)};
\draw plot [smooth,tension = 1] coordinates {(0,0)(1.1,-1)(0,-2)};
\draw plot [smooth,tension = 1] coordinates {(0,-2.2)(1.4,-1.9)(2.2,-3)};
\draw plot [smooth,tension = 1] coordinates {(2.4,-3)(2.8,-1.4)(4,-1)};
\draw plot [smooth,tension = 1] coordinates {(4,0)(2.6,-0.65)(4,-0.8)};
\end{scope}

\begin{scope}[shift={(12,0)}]
\draw (0,0) -- (1.2,-0.5);
\draw (4,0) -- (1.2,-0.5);
\draw (0,-2) -- (1.6,-1.5);
\draw (2.5,-3) -- (1.6, -1.5);
\draw (4,-1) -- (2,-1);
\draw (1.6,-1.5) -- (2,-1);
\draw (1.2,-0.5)--(2,-1);
\end{scope}
\end{tikzpicture}
$$
\end{center}
\caption{Left: the scattering of closed strings ($S^1$) is represented by a point on $\M_{2,4}(\C)$, where marked points, or punctures, are drawn as cusps at infinity.  Center: the scattering of open strings ($[0,1]$) is represented by a point on $\M_{0,5}(\R)$.  Right: as the string length goes to 0, we obtain a Feynman diagram.}
\label{fig:string}
\end{figure}
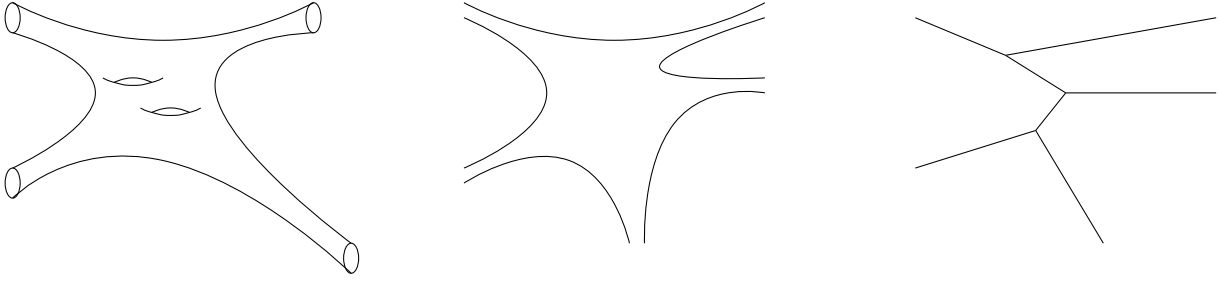
We focus on the open string amplitude, which can be expressed as an integral over $(\M_{0,n})_{>0}$, the connected component of $\M_{0,n}(\R)$ where the $n$ marked points are in the usual cyclic order.  The closure $(\M_{0,n})_{\geq 0}$ of $(\M_{0,n})_{>0}$ in $\bM_{0,n}(\R)$ is a positive geometry \cite{AHLcluster}, with canonical form $\Omega((\M_{0,n})_{>0})$.  This meromorphic top-form is known as the \emph{Parke-Taylor form}, and is often written as follows:
$$
\Omega(\M_{0,n})_{>0}) = \frac{1}{(z_1-z_2)(z_2-z_3) \cdots (z_n-z_1)} \frac{dz_1 dz_2 \cdots dz_n}{\PGL(2)},
$$
where $z_1,z_2,\ldots,z_n$ denotes the $n$ points on $\P^1$; see \cite{LamModuli} for more details.  The \emph{$n$-point tree-level open string amplitude} is the integral
\begin{equation}\label{eq:string}\
I_n(s_{ij}) := \int_{(\M_{0,n})_{>0}} \prod_{i < j} (z_i - z_j)^{\alpha' s_{ij}} \Omega((\M_{0,n})_{>0}),
\end{equation}
viewed as a function of the Mandelstam variables $s_{ij}$, or as a rational function on kinematic space $K_n$.  The parameter $\alpha'$ is the square of the string length.  For $n = 4$, this integral gives the Euler beta function, and we have
$$
I_4|_{s = \alpha' s_{12}, t = \alpha' s_{23}} = B(s,t) = \frac{\Gamma(s) \Gamma(t)}{\Gamma(s+t)}.
$$ 
For higher $n$, the integral \eqref{eq:string} can be analytically continued to a meromorphic function.

The string amplitude exhibits many features of combinatorial geometry.  The closure $(\M_{0,n})_{\geq 0}$ is homeomorphic to the associahedron as a topological stratified space, though it is not linearly embedded as a convex polytope.  This is the context in which Stasheff studied the associahedron.  In particular, $(\M_{0,n})_{>0}$ is homeomorphic to an open ball and one can find parameters $y_1,y_2,\ldots,y_{n-3} \in \R_{>0}$ to parametrize $(\M_{0,n})_{>0}$ so that all the factors $(z_i-z_j)$ in \eqref{eq:string} are polynomials in $y_i$ with positive coefficients.  This \emph{positive parametrization} comes from the coefficient variables of the cluster structure of the Grassmannian $\Gr(2,n)$, and indeed $\M_{0,n}$ can be realized as a cluster configuration space \cite{AHLcluster,AHLT}. 

In the limit of vanishing string length, string theory reduces to quantum field theory:
$$
\lim_{\alpha' \to 0} (\alpha')^{n-3} I_n(s_{ij}) = A_n^{\phi^3}(s_{ij}).
$$
For example, when $n = 4$, the beta function $B(s,t)$ has two simple poles (from $\Gamma(s), \Gamma(t)$) and one zero (from $\Gamma(s+t)$) that pass through the origin, and we have $$\lim_{\alpha' \to 0} \alpha' B(\alpha' s_{12}, \alpha' s_{23}) = \frac{s_{12} + s_{23}}{s_{12} s_{23}} = \frac{1}{s_{12}} + \frac{1}{s_{23}},$$
the two terms corresponding to the two cubic planar trees with four leaves.

How can the combinatorics of the limiting rational function be read off from the integral?
With Arkani-Hamed and He \cite{AHLstringy}, we define a general class of \emph{stringy integrals} with integrands given by products of positive Laurent polynomials and establish a general result of this type:
\begin{equation}\label{eq:stringy}
\lim_{\alpha' \to 0} \left( \int_{\R^d} \prod_{i=1}^d y_i^{\alpha' x_i} \prod_j p_j(\y)^{- \alpha' c_i} \frac{dy_1}{y_1} \cdots \frac{d y_d}{y_d} \right) dx_1 \cdots dx_d = \text{ canonical form of } P,
\end{equation}
where the polytope $P$ is the Minkowski sum of the Newton polytopes of the polynomials $p_j(\y)$ appearing as factors in the integrand of the left hand side (the factors $(z_i-z_j)$ in \eqref{eq:string}).  Thus the positive geometry of the polytope $P$ controls behavior of the stringy integral near the origin.

The limit \eqref{eq:stringy} has an interpretation that is purely geometric.  Saddle-point analysis asks one to consider the behavior of the integral in the neighborhood of the critical points of the integrand.  In the context of the string amplitude \eqref{eq:string}, the critical point equations are the \emph{scattering equations} of Cachazo-He-Yuan \cite{CHYarbitrary}, originally inspired by the twistor string theory of Witten:
$$
d \log \prod_{i < j} (z_i - z_j)^{s_{ij}} = \sum_{i < j} s_{ij} d \log(z_i - z_j) = \sum_i \sum_{j \neq i} \frac{s_{ij}}{z_i - z_j} dz_i = 0
$$
or 
\begin{equation}\label{eq:scateq}
\sum_{j \neq i} \frac{s_{ij}}{z_i - z_j} = 0, \qquad \text{for } i = 1,2,\ldots,n.
\end{equation}
The critical point equations produce the scattering correspondence:
\begin{equation}\label{eq:scatcorr}
\begin{tikzcd}
&\I \arrow{rd}{q} \arrow[swap]{ld}{p} & \\%
\M_{0,n} && K_n
\end{tikzcd}
\end{equation}
where $\I \subset \M_{0,n} \times K_n$ is defined by \eqref{eq:scateq}.  The map $\I \to K_n$ is a map of degree $(n-3)!$.  By restricting to linear subspaces of $K_n$ of dimension $n-3$, the projection $p: \I \to K_n$ can be viewed as a rational map from $\M_{0,n}$ to $K_n$.  Confirming a conjecture in \cite{ABHY}, it is shown in \cite{AHLstringy} that this map sends the ``worldsheet associahedron'' $(\M_{0,n})_{>0}$ onto the interior of the ``kinematic associahedron'' $\Int(\Ass_{n-3})$, and that the pushforward paradigm \cite{ABL} of canonical forms holds:
\begin{equation}\label{eq:pushforward}
p_*  \Omega((\M_{0,n})_{>0}) = \Omega(\Ass_{n-3}).
\end{equation}
This result generalizes to the context of stringy integrals as in \eqref{eq:stringy}; see  \cite{AHLstringy,LamModuli}.

\section{Hyperplane arrangements.}\label{sec:hyper}
The scattering correspondence diagram \eqref{eq:scatcorr} can be generalized to arbitrary very affine varieties (that is, closed subvarieties of an algebraic torus) $U$ replacing the space $\M_{0,n}$.  In algebraic statistics, it is called the likelihood correspondence \cite{HS}.  The kinematic space $K_n$ is replaced by the dual of the Lie algebra of the \emph{intrinsic torus} of $U$.  In this section, we consider this construction for hyperplane arrangements.  

Let $\A = \{H_e  \mid e \in E\}$ be a collection of real hyperplanes in $\P^d$.  We assume that $\A$ is an essential hyperplane arrangement (that is, the normal vectors to the hyperplanes are a spanning set).  We consider the complement $U = \P^d \setminus \A$, and the multi-valued function on $U$,
$$
\varphi := \prod_e f_e^{a_e}
$$
where $f_e$ is a linear form vanishing on $H_e$.  For the function $\varphi$ to be projectively well-defined, we assume that $\sum_e a_e = 0$.  The $1$-form
$$
\omega = d \log \varphi = \sum_{e} a_e \frac{df_e}{f_e}
$$
is an algebraic $d$-form on $U$.

\begin{definition}
The scattering equations for $U$ are the critical point equations $\omega = 0$.  
\end{definition}

We let $\Crit(\omega)$ denote the critical point locus.
The multi-valued function $f$ is known as the \emph{master function} by Varchenko \cite{Varbook}; in the setting of amplitudes, we call it the \emph{scattering potential}.  When $\a \in \Lambda_\C$ is generic the critical point equations consist of $|\chi(U)|$ points, where $\chi(U)$ denotes the Euler characteristic \cite{OT}.  

A chamber of $U$ is a connected component of $U(\R)$.  Each chamber is a polytope and thus a positive geometry.  In the case $U = \M_{0,n}$, the positive part $(\M_{0,n})_{>0}$ is such a chamber.  
\begin{definition}\label{def:CHY}
Let $\Omega_P = h(x_1,\ldots,x_d) dx_1 \ldots dx_d$ be the canonical form of a chamber $P$ in $U(\R)$.  Then the amplitude of $P$ is defined to be
$$
\A_P(\a):= \sum_{p \in \Crit(\omega)} h(p)^2 \det\left(\frac{\partial^2 \log \varphi}{\partial x_i \partial x_j}\right)_p^{-1}.
$$
\end{definition}
While \cref{def:CHY} applies only when the parameters $\a$ are generic, the formula produces a rational function in $\a$. 

We give two alternative descriptions of the rational function $\A_P(\a)$, starting with a combinatorial formula.  Let $L = L(\A)$ denote the lattice of flats of the arrangement $\A$, and for a flat $F$, we define $a_F:= \sum_{e \in F} a_e$.  For a chamber $P$ of $U(\R)$, we let $L(P) \subset L$ denote the subposet of flats that lie on the boundary of $P$.  These are exactly the subspaces spanned by the faces of $P$ (as a projective polytope).  We denote by $\Fl$ and $ \Fl(P)$ the set of maximal flags of flats in $L$ and $L(P)$ respectively.  For a maximal flag $F_\bullet = (\emptyset = F_0 \subset F_1 \subset \cdots \subset F_{d+1} = E)$, we define $a_{F_\bullet}:= \prod_{i=1}^d a_{F_d}$.

\begin{theorem}[\cite{LamMat}]\label{thm:AP}
We have 
\begin{equation}\label{eq:AP}
\A_P(\a) = \sum_{F_\bullet \in \Fl(P)} \frac{1}{a_{F_\bullet}}.
\end{equation}
\end{theorem}

\begin{example} We take $U = \M_{0,5}$, depicted in \cref{fig:M05}.  Here, we have used $\PGL(2)$ to fix $(z_1,z_4,z_5) = (0,1,\infty)$, and the positive part $(\M_{0,5})_{>0}$ is the chamber $P$ given by $z_1 = 0 < z_2 < z_3 < 1 = z_4$.  There are six terms in the summation \eqref{eq:AP}, and we obtain another formula for the $\phi^3$-amplitude:
\begin{multline}\label{eq:temporal}
A_5^{\phi^3} = \frac{1}{s_{12}(s_{12}+s_{13}+s_{23})} +  \frac{1}{s_{12}(s_{12}+s_{34})} +  \frac{1}{s_{34}(s_{12}+s_{34})} \\
\;\;+  \frac{1}{s_{34}(s_{23}+s_{24}+s_{34})} +  \frac{1}{s_{23}(s_{23}+s_{24}+s_{34})} +  \frac{1}{s_{23}(s_{12}+s_{13}+s_{23})} 
\end{multline}
For instance, one flag is $\emptyset \lessdot \{(34)\} \lessdot \{(23),(24),(34)\} \lessdot E = \{(12),(13),(14),(23),(24),(34)\}$, and corresponds to the term with denominator $s_{34}(s_{23}+s_{24}+s_{34})$.  One can check that this summation equals \eqref{eq:5phi3}, noting the relations \eqref{eq:sij}.  Whereas \eqref{eq:5phi3} is a summation over Feynman diagrams that are cubic planar trees, the six-term summation \eqref{eq:temporal} is a summation over \emph{temporal Feynman diagrams} \cite{LamMat}.
\end{example}

\begin{figure}
\begin{center}
$$
\begin{tikzpicture}[scale=1.2,extended line/.style={shorten >=-#1,shorten <=-#1},
 extended line/.default=1cm]
\draw (0,-1) -- (0,2.6);
\draw (1,-1) -- (1,2.6);
\draw (-1,0) -- (2.6,0);
\draw (-1,1) -- (2.6,1);
\draw(-1,-1) -- (2.6,2.6);
\filldraw[fill=red,draw = black] (0,0)--(1,1)--(0,1)--(0,0);
\node at (0,2.8) {\mini $z_2 = 0$};
\node at (1,2.8) {\mini $z_2 = 1$};
\node at (2.6,0.15) {\mini $z_3 = 0$};
\node at (2.6,1.15) {\mini $z_3 = 1$};
\node at (2.6,2.8) {\mini $z_2 = z_3$};
\node[color=blue] at (0.2,2.4) {\mini $s_{12}$};
\node[color=blue] at (1.2,2.4) {\mini $s_{24}$};
\node[color=blue] at (2.7,2.4) {\mini $s_{23}$};
\node[color=blue] at (2.4,-0.15) {\mini $s_{13}$};
\node[color=blue] at (2.4,.85) {\mini $s_{34}$};

\begin{scope}[shift={(7,0)}]
\draw (0,-1) -- (0,2.6);
\draw (1,-1) -- (1,2.6);
\draw (-1,0) -- (2.6,0);
\draw (-1,1) -- (2.6,1);
\draw(-1,-1) -- (2.6,2.6);
\filldraw[color=red] (1,1) circle (0.06cm);
\draw[red,thick] (0,1)--(1,1);
\node at (0,2.8) {\mini $z_2 = 0$};
\node at (1,2.8) {\mini $z_2 = 1$};
\node at (2.6,0.15) {\mini $z_3 = 0$};
\node at (2.6,1.15) {\mini $z_3 = 1$};
\node at (2.6,2.8) {\mini $z_2 = z_3$};
\node[color=blue] at (0.2,2.4) {\mini $s_{12}$};
\node[color=blue] at (1.2,2.4) {\mini $s_{24}$};
\node[color=blue] at (2.7,2.4) {\mini $s_{23}$};
\node[color=blue] at (2.4,-0.15) {\mini $s_{13}$};
\node[color=blue] at (2.4,.85) {\mini $s_{34}$};
\end{scope}
\end{tikzpicture}
$$
\end{center}
\caption{The moduli space $\M_{0,5}$ of $5$ points on $\P^1$ as a hyperplane arrangement complement.  Each of the five hyperplanes has been labeled by one of the Mandelstam variables.  Left: one of the chambers, the positive component, has been highlighted in red.  Right: the flag of lattices $ \{(34)\} \lessdot \{(23),(24),(34)\} $ has been highlighted.}
\label{fig:M05}
\end{figure}
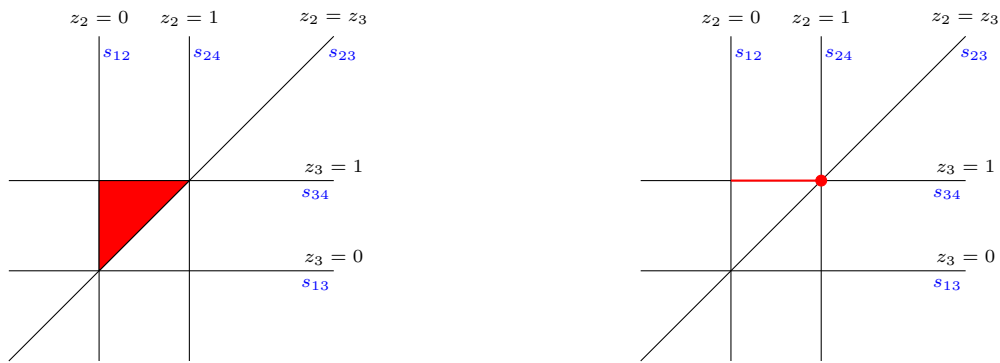

We turn to a cohomological interpretation of the amplitude $\A_P(\a)$.  We have a logarithmic connection $(\O_U,\nabla_\a := d + \omega \wedge)$ on the trivial rank one vector bundle $\O_U$ on $U$, giving a twisted de Rham cohomology group $H^d(U,\nabla_\a)$.  By a result of Esnault--Schechtman--Viehweg \cite{ESV}, this twisted cohomology group can be identified with the cohomology of the Aomoto complex.  This is the complex on the Orlik-Solomon algebra $\OS^\bullet(M)$ (see \cref{sec:matroid}) with differential $\omega \wedge$.  Under a genericity hypothesis \cite[(2.1)]{LamMat}, one has a regularization isomorphism
$$
\reg: H^k(U,\nabla_\a) \stackrel{\cong}{\longrightarrow} H^k_c(U,\nabla_\a)
$$
between twisted cohomology and compactly supported twisted cohomology, inverse to the natural maps between these cohomologies.  We then have a twisted intersection form, first studied by Cho and Matsumoto \cite{CM}: \begin{equation}\label{eq:dR}
\gdRip{\cdot,\cdot}: H^d(U,\nabla^\vee_\a) \otimes H^d(U,\nabla_\a) \xrightarrow{{\rm id} \otimes \reg} H^d(U,\nabla^\vee_\a) \otimes H^d_c(U,\nabla_\a) \xrightarrow{\text{Poincar\'e-Verdier}} \C, 
\end{equation}
where $\nabla^\vee_\a = \nabla_{-\a}$ denotes the dual connection.  In the following formula, we view the holomorphic top-form $\Omega_P$ as an element of both $H^d(U,\nabla^\vee_\a)$ and $H^d(U,\nabla_\a)$. 

\begin{theorem}[\cite{Miz,MHcoh,LamMat}]
The amplitude $\A_P(\a)$ is equal to $\gdRip{\Omega_P, \Omega_P}$.
\end{theorem}

The relationship between the summation over critical points in \cref{def:CHY} and twisted cohomology was established by Matsubara-Heo \cite{MHcoh} in a general setting, and first by Mizera \cite{Miz} for $\M_{0,n}$.  Our work \cite{LamMat} gives a new approach from the perspective of the combinatorics of matroids and hyperplane arrangements.

 More generally, one can consider the twisted cohomological pairings $\gdRip{\Omega_P, \Omega_Q}$ for two chambers $P$ and $Q$.  These pairings are called partial amplitudes, and have a close relation to the \emph{double copy} \cite{BCJ} and \emph{color-kinematics duality} phenomena in scattering amplitudes.  We show in \cite{LamMat} that this pairing is essentially given by the Schechtman--Varchenko contravariant form \cite{SV}.  An earlier but somewhat different relation between twisted cohomology and the Schechtman--Varchenko form was established by Belkale, Brosnan, and Mukhopadhyay \cite{BBM}; see \cite[Remark 2.14]{LamMat}.   
 
 \section{Matroid geometry.}\label{sec:matroid}
The story in \cref{sec:hyper} and in particular the rational function amplitudes $\A_P$ can be generalized to topes $P$ of oriented matroids replacing chambers $P$ of hyperplane arrangement complements.  Let $\M$ be an oriented matroid with underlying matroid $M$ with ground set $E$.  The oriented matroid $\M$ has a set of topes $\T(\M)$.  Each tope $P \in \T(\M)$ is a sign vector $P \in \{+,-\}^E$.  When $\M$ arises from a real hyperplane arrangement, a tope corresponds to a chamber $P$ of $U(\R)$ and the sign vector records which side of each hyperplane $P$ lies on.

The starting point is the work of Eur and the author \cite{EL} generalizing the canonical forms of polytopes to topes of arbitrary oriented matroids.  The canonical form $\Omega_P$ of a tope $P \in \T(\M)$ belongs to the \emph{Orlik-Solomon algebra} $\OS^\bullet(M)$ (our Orlik-Solomon algebras are the \emph{reduced} or \emph{projective} versions).  The Orlik-Solomon algebra is a combinatorial model for the cohomology ring of a hyperplane arrangement complement, and can be defined for an arbitrary (possibly not realizable) matroid $M$.  It is a quotient, by some explicit relations, of the exterior algebra $\Lambda^\bullet(E_0)$ in the generators $(e-e_0)$, $e \in E$, where $e_0 \in E$ is a distinguished element of the ground set.

To state the result, we will use residue maps $\Res_\atom: \OS^\bullet(M) \to \OS^{\bullet-1}(M/\atom)$ between Orlik-Solomon algebras \cite{EL}, where $c$ denotes an atom of the lattice of flats $L(M)$.  These residue maps are an abstract matroid analogue of the residues from \eqref{eq:resdef}.  In the following result, we fix a chirotope $\chi$ for $\M$, used to keep track of signs due to orientations.  Let $r$ denote the rank of $\M$ and $d = r-1$, 
\begin{theorem}\label{thm:EL}
For each tope $P \in \T(\M)$, there exists a distinguished element $\Omega_P \in \OS^d(M)$ satisfying the following recursion:
\begin{enumerate}
\item
If $\M$ has rank $r = 1$, then $\Omega_P = P(e) \chi(e) \in \{1,-1\} \subset \OS^0(M)$ for any $e \in E$.   
\item
If $\M$ has rank $r \geq 1$, then for any atom $\atom$ of $M$, we have
$$
\Res_\atom \Omega_P = \begin{cases} \Omega_{P/\atom} \in \OS(M/\atom) &\mbox{if $\atom \in L(P)$,} \\
0 & \mbox{otherwise.}
\end{cases}
$$
\end{enumerate}
Here, $P/\atom \in \T(\M/\atom)$ is the contracted tope in the contraction $\M/\atom$, and $L(P)$ denotes Las Vergnas face lattice.
\end{theorem}

The Las Vergnas face lattice $L(P)$ of a tope $P$ is the analogue of the face lattice of a polytope.

Before moving on to matroid amplitudes, let us mention some applications of the canonical forms $\Omega_P$.  The most well-known basis of the Orlik-Solomon algebra $\OS^\bullet(M)$ is the no-broken-circuit basis.  Canonical forms provide an alternative basis.  Define a chamber $P$ of the hyperplane arrangement complement $U$ to be \emph{$k$-bounded} if it intersects a fixed but generic codimension $k$ hyperplane $W$, but does not intersect a fixed but generic codimension $k+1$ hyperplane $W' \subset W$.  This notion can be extended to topes of oriented matroids; see \cite{EL}. 

\begin{theorem}\label{thm:bounded}
The canonical forms $\Omega_P$ as $P$ varies over $k$-bounded topes form a basis for $\OS^{d-k}(M)$.
\end{theorem}

\Cref{thm:bounded} gives a concrete realization of a classical result of Greene and Zaslavsky \cite{GZ} stating that the coefficients of the characteristic polynomial of $M$ are equal to the numbers of $k$-bounded chambers.  In the realizable setting, Yoshinaga~\cite{Yos} studied the canonical forms of chambers using a different definition. In very recent work, Brown and Dupont \cite{BDpos} also studied the canonical forms of chambers, and like us, were motivated by positive geometry.

\begin{example}
Let $\M$ be the oriented matroid of rank 2 associated to the arrangement of $n$ points on $\P^1(\R)$, arranged in order, and denote $E = \{e_1,\ldots,e_n\}$.  For each $i = 1,2,\ldots,n$, we have the tope $(i,i+1)$ (an open interval on the real line) with canonical form $e_{i+1} - e_i$, indices taken modulo $n$.  We thus have $n$ canonical forms.  A generic hyperplane of codimension $1$ in this case is simply a point, and we may pick it in the interval $(n,1)$.  Then every other interval is $0$-bounded, and indeed $e_2 -e_1,e_3-e_2,\ldots,e_n-e_{n-1}$ form a basis of $\OS^1(M)$.
\end{example}

Using the canonical forms of topes from \cref{thm:EL}, we can define amplitudes for matroids as follows.  The geometrically defined intersection form $\gdRip{\cdot,\cdot}$ of \eqref{eq:dR} can be lifted to a bilinear form $\ip{\cdot,\cdot}$ on the Orlik-Solomon algebra of $M$, which is purely combinatorial.  This bilinear form is well-defined for any matroid and as shown in \cite{LamMat}, it recovers the matroid generalization \cite{BV} of the Schechtman--Varchenko bilinear form.  We then define the amplitude of a tope $P$ to be
$$
\A_P(\a):= \ip{\Omega_P,\Omega_P},
$$
using the bilinear form on the Orlik-Solomon algebra.  The combinatorial formula \cref{thm:AP} continues to hold for amplitudes of topes.  Furthermore, one has locality and unitarity for these amplitudes.

\begin{theorem}[\cite{LamMat}]\label{thm:APres}
The poles of the rational function $\A_P$ are all simple and can only be along $\{a_F = 0\}$ for connected flats $F$.  We have the recursion
$$
(a_F \A_P)|_{a_F = 0}=  \A_{P|_F} \A_{P|_{E \setminus F}}
$$
where $P|_F$ and $P|_{E \setminus F}$ denotes the restriction and contraction of the tope $P$ in the oriented matroid $\M^F$ (restriction to $F$) and $\M_F$ (contraction by $F$).
\end{theorem}

We summarize the analogy between QFT and matroids in the following table from \cite{LamMat}:
\begin{center}
\begin{tabular}{|c|c|}
\hline

worldsheet & matroid \\
\hline
kinematic space & dual of Lie algebra of intrinsic torus \\
\hline
\# of solutions to scattering equations & beta invariant \\
\hline
Parke-Taylor form & canonical form of a tope \\
\hline
biadjoint scalar partial amplitude & Laplace transform of Bergman fan \\
\hline
inverse string KLT matrix & discrete Laplace transform of Bergman fan\\
\hline
physical poles & connected flats \\
\hline
factorization & deletion-contraction \\
\hline
Feynman diagram & flag of flats \\
\hline
\end{tabular}
\end{center}

\section{Tropical geometry.}\label{sec:tropical}
Let $\Ccal = \bigcup_{k=1}^\infty \C((t^{1/k}))$ be the field of Puiseux series.  Let $U$ be a hyperplane arrangement complement as in \cref{sec:hyper}, and let $n = |E|$ be the number of hyperplanes.  Then $\P^d$ can be identified with a (projective) linear subspace $V$ of the projective space $\P^{n-1}$ so that $U$ is the intersection $V \cap (\C^*)^{n-1}$, where $(\C^*)^{n-1} \subset \P^{n-1}$ is obtained by removing the $n$ coordinate hyperplanes.  Viewing $V$ as the vanishing set of a set of linear equations with constant coefficients, we can define a linear space $V_{\Ccal} \subset (\Ccal^*)^{n-1}$ over the field of Puiseux series.  The tropical linear space $\Trop(V)$ of $V$ is the closure in $\R^{n-1}$ of the set of valuations of points in $V_{\Ccal}$.  It is the underlying set of a polyhedral fan $\Sigma_M$ called the \emph{Bergman fan}, and it only depends on the matroid $M$ of the hyperplane arrangement.  

The Bergman fan below is a slight variant that sits inside $\R^E$, the vector space with basis $\epsilon_e$, $e \in E$.  For $S \subset E$, let $\epsilon_S \in \R^E$ denote the vector $\epsilon_S:= \sum_{s \in S} \epsilon_s$.  For a flag of flats $F_\bullet = (\emptyset \subset F_1 \subset \cdots \subset F_{d+1} = E)$, define
$$
C_{F_\bullet}:= {\rm span}_{\R_{\geq 0}}(\epsilon_{F_1}, \epsilon_{F_2}, \ldots, \epsilon_{F_{d}}).
$$

\begin{definition}
The \emph{Bergman fan} $\Sigma_M$ is the pure $d$-dimensional fan in $\R^E$ whose maximal cones are the cones $C_{F_\bullet}$ for $F_\bullet \in \Fl(M)$.  
\end{definition}

The \emph{Laplace transform} is the integral transform defined as follows.  Given a function $f(\x)$ on $\R^n$, the Laplace transform $\L(f) = \L(f)(\y)$ is the function on $\R^n$
$$
\L(f)(\y) = \int_{\R_{>0}^n} f(\x) \exp(-\x \cdot \y) d^n \x,
$$
which is defined on the domain $\Gamma \subset \R^n$ of convergence of the integral.  The Laplace transform can be generalized to pointed cones $C$ in $\R^n$, by replacing $f(\x)$ with the characteristic function of $C$:
\begin{equation}\label{eq:Lapint}
\L(C)(\y) = \int_C \exp(-\x \cdot \y) d^n \x.
\end{equation}
The integral \eqref{eq:Lapint} converges absolutely for $\y \in \Int(C^*)$ in the interior of the dual cone $C^*$ to a rational function.  This rational function is defined to be the Laplace transform $\L(C)$ of $C$.  In the following, we use the Laplace transform for polyhedral subsets of the Bergman fan: decompose the polyhedral set into cones, and apply the $d$-dimensional Laplace transform.

The \emph{positive Bergman fan} \cite{AKW} $\Sigma_M(P)$ of the tope $P$ is the subfan of $\Sigma_M$ obtained by taking the union of all cones $C_{F_\bullet}$ for $F_\bullet \in \Fl(P)$, together with all the faces of these cones.

\begin{theorem}\label{thm:deRhamfan}
Let $P \in \T(\M)$ be a tope.  Then
\begin{equation}\label{eq:Bergman}
\A_P(\a)= \L(|\Sigma_M(P)|),
\end{equation}
where $|\Sigma_M(P)|$ denotes the underlying set of points of the fan $\Sigma_M(P)$.
\end{theorem}

In the case that $M$ is the graphic matroid of the complete graph, the positive Bergman fan is dual to the face poset of the associahedron.  We have thus come full circle: for the complete graphic matroid, the Laplace transform in \eqref{eq:Bergman} recovers the dual volume polynomial of the associahedron and we obtain \eqref{eq:planarphi3}.  We expect that \cref{thm:deRhamfan} has a wide generalization to tropicalizations of other varieties.

Let us finally mention that our constructions in \cref{sec:polytope}--\cref{sec:tropical} have a variant where the volume function is replaced by taking generating functions of lattice points.  In this alternative story, the de Rham twisted cohomology intersection form in \eqref{eq:dR} is replaced by a Betti twisted homology intersection form; instead of obtaining scattering amplitudes we obtain the \emph{inverse KLT matrix}, an important ingredient in the study of string theory amplitudes.  The Laplace transform in \cref{thm:deRhamfan} is replaced by a discrete Laplace transform.  See \cite{LamMat} for details and references.  In upcoming work \cite{GLX2} joint with Gao and Xue, we consider the lattice point generating function analogue of the dual volume polynomial in \cref{thm:dual}, a discretized version of the canonical form.

\begin{remark}
Tropical geometry also appears in the study of $N=4$ SYM planar amplitudes that we discussed in the earlier half of this survey; see \cite{ALSnon,DFGK,HP}.
\end{remark}

\section{Conclusion.}
The developments surveyed here chart a path from the frontiers of quantum field theory to the frontiers of combinatorial geometry. Scattering amplitudes appear as canonical objects arising from positive geometry. Polytopes, Grassmannians, moduli spaces, hyperplane arrangements, matroids, and tropical varieties each provide a setting where amplitudes manifest as natural geometrically defined functions, unified by the mantra of positivity.

\section*{Acknowledgments.}
I am grateful to my collaborators for the many ideas they have shared and from whom I've learnt so much.  I thank the National Science Foundation for support under grant DMS-2348799.

\bibliographystyle{siamplain}
\bibliography{ICM_references}

\begin{thebibliography}{10}

\bibitem{AGPR}
{\sc N.~Affolter, M.~Glick, P.~Pylyavskyy, and S.~Ramassamy}, {\em
  Vector-relation configurations and plabic graphs}, Selecta Mathematica, 30
  (2023), p.~9.

\bibitem{AKW}
{\sc F.~Ardila, C.~Klivans, and L.~Williams}, {\em The positive {B}ergman
  complex of an oriented matroid}, European J. Combin., 27 (2006),
  pp.~577--591.

\bibitem{ABHY}
{\sc N.~Arkani-Hamed, Y.~Bai, S.~He, and G.~Yan}, {\em {Scattering Forms and
  the Positive Geometry of Kinematics, Color and the Worldsheet}}, JHEP, 05
  (2018), p.~096.

\bibitem{ABL}
{\sc N.~Arkani-Hamed, Y.~Bai, and T.~Lam}, {\em {Positive Geometries and
  Canonical Forms}}, JHEP, 11 (2017), p.~039.

\bibitem{Grassbook}
{\sc N.~Arkani-Hamed, J.~L. Bourjaily, F.~Cachazo, A.~B. Goncharov,
  A.~Postnikov, and J.~Trnka}, {\em {Grassmannian Geometry of Scattering
  Amplitudes}}, Cambridge University Press, 2016.

\bibitem{AHLcluster}
{\sc N.~Arkani-Hamed, S.~He, and T.~Lam}, {\em Cluster configuration spaces of
  finite type}, SIGMA Symmetry Integrability Geom. Methods Appl., 17 (2021),
  pp.~Paper No. 092, 41.

\bibitem{AHLstringy}
{\sc N.~Arkani-Hamed, S.~He, and T.~Lam}, {\em Stringy canonical forms}, J.
  High Energy Phys.,  (2021), pp.~Paper No. 069, 59.

\bibitem{AHLT}
{\sc N.~Arkani-Hamed, S.~He, T.~Lam, and H.~Thomas}, {\em Binary geometries,
  generalized particles and strings, and cluster algebras}, Phys. Rev. D, 107
  (2023), pp.~Paper No. 066015, 8.

\bibitem{ALSnon}
{\sc N.~Arkani-Hamed, T.~Lam, and M.~Spradlin}, {\em Non-perturbative
  geometries for planar {$\mathcal{N} = 4$} {SYM} amplitudes}, J. High Energy
  Phys.,  (2021), pp.~Paper No. 065, 14.

\bibitem{AT}
{\sc N.~Arkani-Hamed and J.~Trnka}, {\em {The Amplituhedron}}, JHEP, 10 (2014),
  p.~030.

\bibitem{BHmom}
{\sc Y.~Bai and S.~He}, {\em {The Amplituhedron from Momentum Twistor
  Diagrams}}, JHEP, 02 (2015), p.~065.

\bibitem{BDMTY}
{\sc V.~Bazier-Matte, N.~Chapelier-Laget, G.~Douville, K.~Mousavand, H.~Thomas,
  and E.~Yildirim}, {\em A{BHY} associahedra and {N}ewton polytopes of
  {$F$}-polynomials for cluster algebras of simply laced finite type}, J. Lond.
  Math. Soc. (2), 109 (2024), pp.~Paper No. e12817, 27.

\bibitem{BBM}
{\sc P.~Belkale, P.~Brosnan, and S.~Mukhopadhyay}, {\em Hyperplane arrangements
  and tensor product invariants}, Michigan Math. J., 68 (2019), pp.~801--829.

\bibitem{BCJ}
{\sc Z.~Bern, J.~J.~M. Carrasco, and H.~Johansson}, {\em {Perturbative Quantum
  Gravity as a Double Copy of Gauge Theory}}, Phys. Rev. Lett., 105 (2010),
  p.~061602.

\bibitem{BlKa}
{\sc A.~M. Bloch and S.~N. Karp}, {\em Gradient flows, adjoint orbits, and the
  topology of totally nonnegative flag varieties}, Comm. Math. Phys., 398
  (2023), pp.~1213--1289.

\bibitem{BCFW}
{\sc R.~Britto, F.~Cachazo, B.~Feng, and E.~Witten}, {\em Direct proof of the
  tree-level scattering amplitude recursion relation in {Y}ang-{M}ills theory},
  Phys. Rev. Lett., 94 (2005), pp.~181602, 4.

\bibitem{BDpos}
{\sc F.~Brown and C.~Dupont}, {\em Positive geometries and canonical forms via
  mixed hodge theory}, \arxiv{2501.03202},  (2025).

\bibitem{BV}
{\sc T.~Brylawski and A.~Varchenko}, {\em The determinant formula for a matroid
  bilinear form}, Adv. Math., 129 (1997), pp.~1--24.

\bibitem{CHYarbitrary}
{\sc F.~Cachazo, S.~He, and E.~Y. Yuan}, {\em {Scattering of Massless Particles
  in Arbitrary Dimensions}}, Phys. Rev. Lett., 113 (2014), p.~171601.

\bibitem{CM}
{\sc K.~Cho and K.~Matsumoto}, {\em Intersection theory for twisted
  cohomologies and twisted {R}iemann's period relations. {I}}, Nagoya Math. J.,
  139 (1995), pp.~67--86.

\bibitem{DPSV}
{\sc S.~De, D.~Pavlov, M.~Spradlin, and A.~Volovich}, {\em From {F}eynman
  diagrams to the amplituhedron: a gentle review}, Matematiche (Catania), 80
  (2025), pp.~233--254.

\bibitem{Dix}
{\sc L.~J. Dixon}, {\em Scattering amplitudes: the most perfect microscopic
  structures in the universe}, J. Phys. A, 44 (2011), pp.~454001, 24.

\bibitem{DFGK}
{\sc J.~Drummond, J.~Foster, {\"O}.~G{\"u}rdo{\u g}an, and C.~Kalousios}, {\em
  {Tropical Grassmannians, cluster algebras and scattering amplitudes}},
  (2019).

\bibitem{DHP}
{\sc J.~Drummond, J.~Henn, and J.~Plefka}, {\em Yangian symmetry of scattering
  amplitudes in {$N=4$} super {Y}ang-{M}ills theory}, J. High Energy Phys.,
  (2009), pp.~046, 23.

\bibitem{EH}
{\sc H.~Elvang and Y.-t. Huang}, {\em Scattering amplitudes in gauge theory and
  gravity}, Cambridge University Press, Cambridge, 2015.

\bibitem{ESV}
{\sc H.~Esnault, V.~Schechtman, and E.~Viehweg}, {\em Cohomology of local
  systems on the complement of hyperplanes}, Invent. Math., 109 (1992),
  pp.~557--561.

\bibitem{EL}
{\sc C.~Eur and T.~Lam}, {\em Canonical forms for oriented matroids},
  \arxiv{2502.20782},  (2025).

\bibitem{ELPSTW}
{\sc C.~Even-Zohar, T.~Lakrec, M.~Parisi, M.~Sherman-Bennett, R.~Tessler, and
  L.~Williams}, {\em Cluster algebras and tilings for the m=4 amplituhedron},
  \arxiv{2504.01217},  (2023).

\bibitem{FeSa}
{\sc C.~Fevola and A.-L. Sattelberger}, {\em Algebraic and positive geometry of
  the universe: from particles to galaxies}, Notices Amer. Math. Soc., 72
  (2025), pp.~808--817.

\bibitem{GKL3}
{\sc P.~Galashin, S.~N. Karp, and T.~Lam}, {\em Regularity theorem for totally
  nonnegative flag varieties}, J. Amer. Math. Soc., 35 (2022), pp.~513--579.

\bibitem{GKL1}
{\sc P.~Galashin, S.~N. Karp, and T.~Lam}, {\em The totally nonnegative
  {G}rassmannian is a ball}, Adv. Math., 397 (2022), pp.~Paper No. 108123, 23.

\bibitem{GLparity}
{\sc P.~Galashin and T.~Lam}, {\em Parity duality for the amplituhedron},
  Compos. Math., 156 (2020), pp.~2207--2262.

\bibitem{GLX}
{\sc Y.~Gao, T.~Lam, and L.~Xue}, {\em Dual mixed volume}, \arxiv{2410.21688},
  (2024).

\bibitem{GLX2}
{\sc Y.~Gao, T.~Lam, and L.~Xue}, {\em in preparation},  (2025).

\bibitem{GZ}
{\sc C.~Greene and T.~Zaslavsky}, {\em On the interpretation of {W}hitney
  numbers through arrangements of hyperplanes, zonotopes, non-{R}adon
  partitions, and orientations of graphs}, Trans. Amer. Math. Soc., 280 (1983),
  pp.~97--126.

\bibitem{HP}
{\sc N.~Henke and G.~Papathanasiou}, {\em How tropical are seven- and
  eight-particle amplitudes?}, J. High Energy Phys.,  (2020), pp.~005, 49.

\bibitem{HePl}
{\sc J.~M. Henn and J.~C. Plefka}, {\em Scattering amplitudes in gauge
  theories}, vol.~883 of Lecture Notes in Physics, Springer, Heidelberg, 2014.

\bibitem{Hod}
{\sc A.~Hodges}, {\em {Eliminating spurious poles from gauge-theoretic
  amplitudes}}, JHEP, 05 (2013), p.~135.

\bibitem{HS}
{\sc J.~Huh and B.~Sturmfels}, {\em Likelihood geometry}, in Combinatorial
  algebraic geometry, vol.~2108 of Lecture Notes in Math., Springer, Cham,
  2014, pp.~63--117.

\bibitem{KW}
{\sc S.~N. Karp and L.~K. Williams}, {\em The {$m=1$} amplituhedron and cyclic
  hyperplane arrangements}, Int. Math. Res. Not. IMRN,  (2019), pp.~1401--1462.

\bibitem{KLS}
{\sc A.~Knutson, T.~Lam, and D.~E. Speyer}, {\em Positroid varieties: juggling
  and geometry}, Compos. Math., 149 (2013), pp.~1710--1752.

\bibitem{Lamm=2}
{\sc T.~Lam}, {\em On the face stratification of the $m=2$ amplituhedron},
  Combinatorial Theory, to appear.

\bibitem{LamAS}
{\sc T.~Lam}, {\em Affine {S}tanley symmetric functions}, Amer. J. Math., 128
  (2006), pp.~1553--1586.

\bibitem{Lamampl}
{\sc T.~Lam}, {\em Amplituhedron cells and {S}tanley symmetric functions},
  Comm. Math. Phys., 343 (2016), pp.~1025--1037.

\bibitem{LamCDM}
{\sc T.~Lam}, {\em Totally nonnegative {G}rassmannian and {G}rassmann
  polytopes},  (2016), pp.~51--152.

\bibitem{LamMat}
{\sc T.~Lam}, {\em Matroids and amplitudes}, \arxiv{2412.06705},  (2024).

\bibitem{LamPosGeom}
{\sc T.~Lam}, {\em An invitation to positive geometries}, in Open problems in
  algebraic combinatorics, vol.~110 of Proc. Sympos. Pure Math., Amer. Math.
  Soc., Providence, RI, [2024] \copyright 2024, pp.~159--179.

\bibitem{LamModuli}
{\sc T.~Lam}, {\em Moduli spaces in positive geometry}, Matematiche (Catania),
  80 (2025), pp.~17--101.

\bibitem{LLMS}
{\sc T.~Lam, L.~Lapointe, J.~Morse, and M.~Shimozono}, {\em Affine insertion
  and {P}ieri rules for the affine {G}rassmannian}, Mem. Amer. Math. Soc., 208
  (2010), pp.~xii+82.

\bibitem{Lee}
{\sc S.~J. Lee}, {\em Pieri rule for the affine flag variety}, Adv. Math., 304
  (2017), pp.~266--284.

\bibitem{LusTP}
{\sc G.~Lusztig}, {\em Total positivity in reductive groups}, in Lie theory and
  geometry, vol.~123 of Progr. Math., Birkh\"auser Boston, Boston, MA, 1994,
  pp.~531--568.

\bibitem{Mar}
{\sc M.~Marcolli}, {\em Feynman motives}, World Scientific Publishing Co. Pte.
  Ltd., Hackensack, NJ, 2010.

\bibitem{MHcoh}
{\sc S.-J. Matsubara-Heo}, {\em Localization formulas of cohomology
  intersection numbers}, J. Math. Soc. Japan, 75 (2023), pp.~909--940.

\bibitem{Miz}
{\sc S.~Mizera}, {\em Scattering amplitudes from intersection theory}, Phys.
  Rev. Lett., 120 (2018), pp.~141602, 6.

\bibitem{OT}
{\sc P.~Orlik and H.~Terao}, {\em The number of critical points of a product of
  powers of linear functions}, Invent. Math., 120 (1995), pp.~1--14.

\bibitem{PT}
{\sc S.~J. Parke and T.~R. Taylor}, {\em Amplitude for n-gluon scattering},
  Physical Review Letters., 56 (1986), pp.~2459--2460.

\bibitem{PosTP}
{\sc A.~Postnikov}, {\em {Total positivity, Grassmannians, and networks}},
  (2006), \url{https://arxiv.org/abs/math/0609764}.

\bibitem{PGvolume}
{\sc K.~Ranestad, F.~Russo, B.~Sturmfels, and S.~Telen}, eds., {\em Special
  Issue on ``Positive Geometry''}, vol.~80, Matematiche (Catania), 2025.

\bibitem{RST}
{\sc K.~Ranestad, R.~Sinn, and S.~Telen}, {\em Adjoints and canonical forms of
  tree amplituhedra}, Math. Scand., 130 (2024), pp.~433--466.

\bibitem{RaStTe}
{\sc K.~Ranestad, B.~Sturmfels, and S.~Telen}, {\em What is positive
  geometry?}, Matematiche (Catania), 80 (2025), pp.~3--16.

\bibitem{SV}
{\sc V.~V. Schechtman and A.~N. Varchenko}, {\em Arrangements of hyperplanes
  and {L}ie algebra homology}, Invent. Math., 106 (1991), pp.~139--194.

\bibitem{Varbook}
{\sc A.~Varchenko}, {\em Multidimensional hypergeometric functions and
  representation theory of {L}ie algebras and quantum groups}, vol.~21 of
  Advanced Series in Mathematical Physics, World Scientific Publishing Co.,
  Inc., River Edge, NJ, 1995.

\bibitem{Yos}
{\sc M.~Yoshinaga}, {\em The chamber basis of the {O}rlik-{S}olomon algebra and
  {A}omoto complex}, Ark. Mat., 47 (2009), pp.~393--407.

\end{thebibliography}
\end{document}